\documentclass[%
  twocolumn
]{mpi2015-cscpreprint}

\usepackage[american]{babel}

\usepackage{graphicx}

\usepackage{amsmath}
\usepackage{amssymb}
\usepackage{amsthm}
\usepackage{algorithmic}
\usepackage{algorithm}
\usepackage{verbatim}
\usepackage{xspace,extarrows,fdsymbol}
\renewcommand{\algorithmicrequire}{\textbf{Input:}}
\renewcommand{\algorithmicensure}{\textbf{Output:}}

\usepackage{tikz}
\usetikzlibrary{positioning}

\usepackage{pgfplots}
\pgfplotsset{compat = 1.13,
	colormap name = viridis,
	unbounded coords = jump}

\tikzset{every picture/.style={/utils/exec={\normalfont}}}

\usepackage{caption}
\usepackage{subcaption}
\definecolor{myRed}{HTML}{E34A33}
\definecolor{myBlue}{HTML}{0571B0}
\definecolor{myBrown}{HTML}{A6611A}

\usepackage{url}
\usepackage[nameinlink]{cleveref}

\crefformat{equation}{\textup{#2(#1)#3}}
\crefrangeformat{equation}{\textup{#3(#1)#4--#5(#2)#6}}
\crefmultiformat{equation}{\textup{#2(#1)#3}}{ and \textup{#2(#1)#3}}
{, \textup{#2(#1)#3}}{, and \textup{#2(#1)#3}}
\crefrangemultiformat{equation}{\textup{#3(#1)#4--#5(#2)#6}}%
{ and \textup{#3(#1)#4--#5(#2)#6}}{, \textup{#3(#1)#4--#5(#2)#6}}
{, and \textup{#3(#1)#4--#5(#2)#6}}

\usepackage{enumitem}

\usepackage{todonotes}

\def\IC{{\mathbb C}}

\def\IL{{\mathbb L}}

\def\IT{{\mathbb T}}

\newcommand{\bA}{{\mathbf A}}

\newcommand{\bG}{{\mathbf G}}

\newcommand{\bK}{{\mathbf K}}

\newcommand{\bM}{{\mathbf M}}

\newcommand{\bI}{{\mathbf I}}

\newcommand{\bX}{{\mathbf X}}
\newcommand{\bx}{{\mathbf x}}

\newcommand{\bu}{{\mathbf u}}

\newcommand{\bU}{{\mathbf U}}

\newcommand{\bc}{{\mathbf c}}
\newcommand{\bfe}{{\mathbf e}}

\newcommand{\bb}{{\mathbf b}}

\newcommand{\bh}{{\mathbf h}}

\newcommand{\bv}{{\mathbf v}}
\newcommand{\bw}{{\mathbf w}}

\newcommand{\bfz}{{\mathbf 0}}

\newcommand{\cN}{ {\cal N} }

\newcommand{\bPhi}{ \boldsymbol{\Phi} }

\newcommand{\bhA}{\hat{\mathbf A}}
\newcommand{\bhb}{\hat{\mathbf b}}
\newcommand{\bhc}{\hat{\mathbf c}}

\def\iop{{\textrm i}}

\theoremstyle{plain}\newtheorem{theorem}{Theorem}
\theoremstyle{plain}\newtheorem{corollary}{Corollary}
\theoremstyle{plain}\newtheorem{proposition}{Proposition}
\theoremstyle{plain}\newtheorem{lemma}{Lemma}
\theoremstyle{definition}\newtheorem{remark}{Remark}

\newcommand{\eqdef}{\xlongequal{\text{def}}}
\newcommand{\hh}{\hat{h}}
\newcommand{\bhh}{\hat{\bh}}

\newcommand{\s}{{\xi}}
\newcommand{\sh}{\hat{\s}}
\newcommand{\bolds}{\boldsymbol{\s}}
\newcommand{\boldsh}{\hat{\bolds}}
\newcommand{\lqoaaa}{\textsf{AAA-LQO}\xspace}
\newcommand{\aaa}{\textsf{AAA}\xspace}
\newcommand{\ls}{\textsf{LS}\xspace}
\newcommand{\lqo}{\textsf{LQO}\xspace}

\def\J{{\mathcal J}}

\def\Lonetwo{\IL^{\hspace{-0.5ex}\mbox{\scriptsize{(1,2)}}}}

\def\Ltwoone{\IL^{\hspace{-0.5ex}\mbox{\scriptsize{(2,1)}}}}

\def\Ltwod{\IL^{\hspace{-0.5ex}\mbox{\scriptsize{\textrm (2,2)}}}}

\def\got{\hat{\bh}^{\hspace{0.1ex}\mbox{\scriptsize{\textrm (1,2)}}}}

\def\gto{\hat{\bh}^{\hspace{-0.0ex}\mbox{\scriptsize{\textrm (2,1)}}}}

\def\gtt{\hat{\bh}^{\hspace{-0.0ex}\mbox{\scriptsize{\textrm (2,2)}}}}

\def\hij{\hat{h}_{i,j}^{\hspace{-0.0ex}\mbox{\scriptsize{\textrm (2,2)}}}}

\def\gij{\hat{h}_{i,j}^{\hspace{-0.0ex}\mbox{\scriptsize{\textrm (1,2)}}}}

\def\fji{\hat{h}_{j,i}^{\hspace{-0.0ex}\mbox{\scriptsize{\textrm (2,1)}}}}

\begin{document}
  

\title{Data-driven modeling of linear dynamical systems with quadratic output in the AAA framework}
  
\author[$\ast$]{Ion Victor Gosea}
\affil[$\ast$]{Max Planck Institute for Dynamics of Complex Technical Systems,
	Sandtorstr. 1, 39106 Magdeburg, Germany.\authorcr
  \email{gosea@mpi-magdeburg.mpg.de}, \orcid{0000-0003-3580-4116}}
  
\author[$\dagger$]{Serkan Gugercin}
\affil[$\dagger$]{Department of Mathematics and Computational Modeling and Data
	Analytics Division, Academy of Integrated Science, Virginia Tech, Blacksburg,
	VA 24061, USA.\authorcr
  \email{gugercin@vt.edu}, \orcid{0000-0003-4564-5999}}
  
\shorttitle{The AAA framework for linear dynamical systems with quadratic output}
\shortauthor{I. V. Gosea, S. Gugercin}
\shortdate{}
  
\keywords{\small data-driven modeling, nonlinear dynamics, rational approximation, barycentric form, interpolation. \normalsize}

  
\abstract{%
We extend the \aaa (Adaptive-Antoulas-Anderson) algorithm to develop a data-driven modeling framework for linear systems with quadratic output (\lqo). Such systems are characterized by two transfer functions: one corresponding to the linear part of the output and another one to the quadratic part. We first establish the joint barycentric representations and the  interpolation theory  for the two transfer functions of \lqo systems. This analysis leads to the proposed \lqoaaa algorithm. We show that by interpolating the transfer function values on a subset of samples together with imposing a least-squares minimization on the rest, we construct reliable data-driven \lqo models. Two numerical test cases illustrate the efficiency of the proposed method.}
  
\maketitle

  
\section{Introduction}
\label{sec:intro}

Model order reduction (\textsf{MOR}) is used to approximate large-scale dynamical systems with  smaller ones that ideally have similar response characteristics to the original one. This has been an active research area and  many approaches to \textsf{MOR} have been proposed. We refer the reader to  \cite{ACA05,BBF14,AntBG20,quarteroni2015reduced,siammorbook2017,SCARCIOTTI17}
and the references therein  for an overview of \textsf{MOR} methods for both linear and nonlinear dynamical systems.

\textsf{MOR}, as the name implies, assumes access to a full order model to be reduced; in most cases, in the form of a state-space formulation obtained via, e.g., a spatial discretization of the underlying partial differential equations. Then, the reduced order quantities are computed via an explicit projection of the full-order quantities. However, in some cases, access to full order  dynamics is not available. Instead, one has  access to a collection of input/output measurements. In this case, the goal is to construct the approximation directly from data, which we refer to as data-driven modeling. This is the framework we consider in this paper.

Specifically, we focus on data-driven modeling  of linear dynamical systems with quadratic output (\lqo). In our formulation, data correspond to  frequency domain samples of the input/output mapping of the underlying \lqo system, in the form of samples  of its two transfer functions; the first transfer function being a single-variable one and the second a bivariate one.
For this data set, the proposed framework first develops the barycentric rational interpolation theory  for \lqo systems to interpolate a subset of the data and  and then 
extends the \aaa algorithm \cite{NST18} to this setting by minimizing a least-square measure in the remaining data.

We note that system identification of general nonlinear systems has been a popular topic. In particular, we mention here the special case of identifying linear systems with nonlinear output or input functions, e.g., the so-called Wiener \cite{Wiener58} and Hammerstein models, respectively.  Significant effort has been allocated for identification of such models; see, e.g., \cite{juditsky95}, \cite{giri10} and the references therein. Nevertheless, the methods previously mentioned are based in the time domain, while in this paper we focus on frequency domain data. We point out that the frequency-data based Loewner framework was recently extended to identifying Hammerstein models in \cite{morKarGA21}.

The rest of the paper is organized as follows: 
We discuss \lqo systems and their transfer functions
in \Cref{sec:lqo},  followed by a review of barycentric rational approximation and the \aaa algorithm in \Cref{sec:aaa}. Next, we develop the theory for
barycentric representation and multivariate interpolation for \lqo systems in \Cref{sec:lqobary}. 
Based on this analysis, in
\Cref{sec:prop_meth}, we present 
the proposed algorithm, \lqoaaa,  for data-driven modeling of \lqo systems.
The numerical experiments are given in \Cref{sec:numerics} followed by the conclusions  in \Cref{sec:conc}.

\section{Linear systems with quadratic output}
\label{sec:lqo} 
In state-space form,  linear dynamical systems with quadratic output (\lqo systems) are  described as
\begin{align}\label{eq:def_LQO}
\Sigma_{\lqo}: \begin{cases}  \dot\bx(t)=\bA\bx(t)+\bb\hspace{0.25ex} u(t),\\
\hspace{0.5mm} y(t)=\bc^T\bx(t)
+{\bK \big{[} \bx(t) \otimes \bx(t) \big{]}},
\end{cases}
\end{align}
where  $\bA \in \mathbb{R}^{\cN \times \cN}$, $\bb, \bc \in \mathbb{R}^{\cN}$, $\bK \in \mathbb{R}^{1 \times \cN^2}$, 
and the symbol $\otimes$ denotes the Kronecker product, i.e., for the vector $\bx = [x_1 \ x_2 \ \cdots \ x_\cN]^T \in \mathbb{R}^{\cN^2}$, we have
$$
\bx \otimes \bx = [x_1^2 \ \ x_1 x_2 \ \ x_1 x_3 \ \ \cdots \ \ x_1 x_\cN \ \ \cdots x_\cN^2]^T \in \mathbb{R}^{\cN^2}.
$$
The quadratic part of the output in 
\cref{eq:def_LQO}, 
$\bK \big{[} \bx(t) \otimes \bx(t) \big{]}$,  can be rewritten as $\bx^T(t) \bM \bx(t)$ with $ \bM \in \mathbb{R}^{\cN \times \cN}$ and $\bK = \text{vec}(\bM))$ where 
$\text{vec}$ denotes the vectorization
operation. In some cases, in \cref{eq:def_LQO} we have $\bc = \mathbf{0}$, and thus the output has only the quadratic term.

Several projection-based MOR methodologies have been already proposed for \lqo systems. More precisely, balanced truncation-type methods were considered in \cite{BM10,PN19,BGP19}, while interpolation-based methods were used in \cite{BNLM12,GA19}. All these methods are intrusive, meaning that, they explicitly work with the state-space matrices $\bA,\bb,\bc$ and $\bK$ in \cref{eq:def_LQO}.

The main goal of this work is to develop a \emph{data-driven} modeling framework for \lqo systems where only input-output measurements, in the form of transfer function evaluations, are needed as opposed to the internal state-space representation. Therefore, our first goal is to derive transfer  functions for this special class of dynamical systems.

\subsection{The transfer functions of \lqo systems}

Many classes of nonlinear systems can be represented in the time domain by generalized kernels as presented in the classical Wiener or Volterra series representations. Generically, infinite number of kernels appear in such series, corresponding to each homogeneous subsystem. For more details we refer the reader to \cite{Wiener58,Rugh81}

For the \lqo system  \cref{eq:def_LQO}, the nonlinearity is present in the state-to-output equation only and one can write the input-output mapping of the system  in the frequency domain using two transfer functions
\begin{enumerate}
	\item one corresponding to the linear part of the output, i.e., $y_1(t) = \bc^T \bx(t)$;
	\item one corresponding to the quadratic part  of the output, i.e., $y_2(t)  = \bK (\bx(t) \otimes \bx(t))$.
\end{enumerate}
These transfer functions were recently derived in \cite{GA19} using their time-domain representations. In the next result, we introduce  and re-derive them for the completeness of the paper and to illustrate to the reader how they naturally appear.
\begin{lemma} \label{lem:tfs}
	Consider the \lqo system in \cref{eq:def_LQO} with $\bx(0) = \mathbf{0}$. Let the input $u(t)$ be a sum of the 
	$J$ harmonic terms, i.e.,
	\begin{equation} \label{inputu}
	u(t) = \sum_{j  =1}^J e^{\iop \omega_j t },~~~~\mbox{where}~~\omega_j > 0~~\mbox{for}~~j=1,2,\ldots,J,~~
	\end{equation}
	and $\iop^2 = -1$.  Then, the output $y(t)$ is given by
	\begin{equation} \label{yh1h2}
	y(t) =
	\sum_{j=1}^J H_1(\iop \omega_j) e^{\iop \omega_j t} +  \sum_{j=1}^J \sum_{\ell=1}^J H_2(\iop \omega_j,\iop \omega_\ell) e^{\iop (\omega_j + \omega_\ell)t},
	\end{equation}
	where 
	\begin{equation} \label{eq:H1}
	H_1(s) = \bc^T(s\bI_n-\bA)^{-1} \bb
	\end{equation}
	is the single-variable rational transfer function corresponding to $y_1(t)$ and 
	\begin{equation} \label{eq:H2}
	H_2(s,z) = \bK \Big[ (s\bI_n-\bA)^{-1} \bb \otimes (z\bI_n-\bA)^{-1} \bb \Big]
	\end{equation} 
	is the two-variable rational transfer function corresponding to $y_2(t)$ with
	$\bI_n$ denoting the identity matrix of size $ \cN\times \cN$.
\end{lemma}
\begin{proof} For the input $u(t)$ in \cref{inputu} with $\bx(0) = \mathbf{0}$,
	the solution of the linear state-equation in \cref{eq:def_LQO} in steady-state can be written as a sum of scaled complex exponential functions as
	\begin{equation}\label{x_sum}
	\bx(t) = \sum_{j=1}^J \bG_1(\iop \omega_j) e^{\iop \omega_j t},
	\end{equation}
	where
	$\bG_1(s) = (s\bI-\bA)^{-1} \bb$.
	Substituting \cref{x_sum} into the output equation of \cref{eq:def_LQO}, we obtain
	\begin{align}
	\begin{split}
	y(t) &= \bc^T \sum_{j=1}^J \bG_1(\iop \omega_j) e^{\iop \omega_j t} \\
	+& \bK \Big{[} \sum_{j=1}^J \bG_1(\iop \omega_j) e^{\iop \omega_j t} \Big{]} \otimes \Big{[} \sum_{\ell=1}^J \bG_1(\iop \omega_\ell) e^{\iop \omega_\ell t} \Big{]} \\
	&=  \sum_{j=1}^J \bc^T \bG_1(\iop \omega_j) e^{j \omega_j t} \\ &+ \sum_{j=1}^J \sum_{\ell=1}^J \bK \big{[} \bG_1(\iop \omega_j)  \otimes \bG_1(\iop \omega_\ell) \big{]} e^{\iop (\omega_j + \omega_\ell)t}. 
	\end{split}
	\end{align}
	Substituting $\bG_1(s) = (s\bI-\bA)^{-1} \bb$ 
	back into the last equation 
	yields the desired result\cref{yh1h2} with $H_1(s)$ and  
	$H_2(s,z)$ as defined in \cref{eq:H1}
	and \cref{eq:H2}. 
\end{proof}

\Cref{lem:tfs} shows that the \lqo system \cref{eq:def_LQO} is characterized by two transfer functions, namely $H_1(s)$ (corresponding to the linear component $y_1(t)$ in the output)  and $H_2(s,z)$ (corresponding to the quadratic linear  $y_2(t)$ in the output).
As in the classical linear case,  $H_1(s)$ is a
rational function in single variable. On the other hand, 
$H_2(s,z)$ is also a rational function, but of two variables. These two transfer functions that fully describe the \lqo system  \cref{eq:def_LQO} will play the fundamental role in our analysis to extend  barycentric interpolation and \aaa to the \lqo setting. Before we establish the theory for \lqo systems, we will briefly review the \aaa algorithm for linear systems in \Cref{sec:aaa}. 

\begin{remark}
	In the proposed framework, we will require sampling the two transfer functions $H_1(s)$ and $H_2(s,z)$. As it is shown in \Cref{lem:tfs}, this could be achieved by exciting the system (as a black box) with purely oscillatory control inputs and measuring the outputs, and performing a Fourier transformation. For more details on such procedures in similar settings, we refer the reader to
	\cite{morKarGA20}.
	We also note that \cite{Tick61} examines systems described by two time-domain kernels together with their Fourier transformations (deemed as transfer functions) and their measurements. Even though no explicit representation of these functions are considered in terms of a state-space realization, those ideas also equally apply to sample $H_1(s)$ and $H_2(s,z)$.
\end{remark}

\begin{remark}
	Note that in the special case for which it holds that $\bK = \alpha ( \bc^T \otimes \bc^T)$, we obtain $y_2(t) = \alpha y_1^2(t)$ where $\alpha$ is a scalar. Therefore, in this case the output $y(t)$ is a quadratic polynomial in the linear output $y_1(t)$ and the \lqo model can be interpreted as a Wiener model \cite{Wiener58}. However, our focus here is on the general case  of $\lqo$ systems without this special case.
\end{remark}

\section{Barycentric rational approximation for linear systems and the \aaa algorithm}
\label{sec:aaa}

For an underlying function 
$H(\cdot): \IC \to \IC$, e.g., transfer function of a single-input/single-output 
(SISO)
linear dynamical system, assume the following set of measurements:
\begin{align}  \label{fulldata}
\{H(s_i)\} \in \mathbb{C}~~\mbox{where}~~s_i \in \mathbb{C}\quad\mbox{for}~~ i=1,2,\ldots,N_s.
\end{align}
Partition the sampling points into two disjoint sets: 
\begin{equation}
\begin{array}{rcl}
\{ s_1, \dots, s_{N_s} \}  & \hspace{-2ex}= &\hspace{-2ex}
\{~\s_1, \dots, \s_{n} ~\}  \cup  
\{~\sh_1, \dots, \sh_{N_s-n}~\}\\ &\hspace{-2ex} \eqdef &\hspace{-2ex}
\{\hspace{3.5ex}~\bolds \hspace{4.75ex}~\cup\hspace{4.5ex}~\boldsh~\hspace{7ex}\}.
\end{array}
\label{xixihat}
\end{equation}
We will clarify later how this partitioning is chosen.
Based on  \cref{xixihat}, define the sampled values
\begin{align}
\begin{array}{lcr}
h_i & \hspace{-2ex}\eqdef &\hspace{-1.5ex}H(\s_i)~~~\mbox{for}~~~i=1,2,\ldots,n,~~\mbox{and}\\
\hh_i & \hspace{-2ex}\eqdef &\hspace{-1.5ex}H(\sh_i)~~~\mbox{for}~~~i=1,2,\ldots,N_s-n,
\end{array} \label{lindata}
\end{align}
and the corresponding data sets
\begin{align}
\bh \eqdef \{h_1,\ldots,h_{n}\}~~~\mbox{and}~~~
\bhh \eqdef \{\hh_1,\ldots,\hh_{N_s-n}\}.
\label{lindataset}
\end{align}
Define the rational function $r(s)$ in barycentric form \cite{berrut06barycentric}, a numerically stable representation of rational functions\footnote{With the addition of $1$ to the denominator, we guarantee that $r(s)$ is a strictly proper rational function with a numerator degree $n-1$ and the denominator degree $n$. This is done in the anticipation of the dynamical system in \cref{eq:def_LQO} we aim to approximate  where there will be no direct input-to-output mapping. This is not a restriction, and  the numerator and denominator degrees can be chosen in a different way \cite{berrut06barycentric,NST18}.}:
\begin{equation}\label{eq:bary_rep}
r(s) =  \dfrac{p(s)}{q(s)} = \frac{\displaystyle\sum_{k=1}^n \frac{w_k h_k}{s - \s_k}}{\displaystyle 1+ \sum_{k=1}^n\frac{w_k}{s - \s_k}},
\end{equation}
where
$\s_k \in \mathbb{C}$ are the \emph{sampling (support) points} and
the \emph{weights} $w_k \in \mathbb{C}$ are to be determined.  By construction, the degree-($n-1$) rational function
$r(s)$ in   \cref{eq:bary_rep} is a rational interpolant at the support point set $\bolds$, i.e.,
\begin{align}  \label{singleinterp}
r(\s_k) = h_k \quad\mbox{for}\quad  k=1,2,\ldots,n,
\end{align}
assuming $w_k \neq 0$. Then, the freedom in choosing the weights $\{w_k\}$ can be used to match the remaining the  data $\bhh$ in an appropriate measure.

Assuming enough degrees of freedom,   
\cite{AA86} chooses the weights $\{w_k\}$ to enforce interpolation of $\bhh$ as well, by computing the null space of the corresponding divided difference matrix, thus obtaining a degree-($n-1$) rational function interpolating the full data \cref{fulldata}. We skip the details for the conditions to guarantee the existence and uniqueness of such a rational interpolant  and  refer the reader to \cite{AA86,AntBG20} for details. 

The \aaa (Adaptive-Antoulas-Anderson) algorithm~\cite{NST18}, on the other hand, elegantly combines interpolation and least-squares (\ls) fitting. In the barycentric form \cref{eq:bary_rep},
which interpolates the data $\bh$ by construction,
\aaa chooses the weights $\{w_k\}$ to minimize a \ls error over the data  $\bhh$. Note that the \ls problem over $\bhh$ is nonlinear in the weights $\{w_k\}$ since these weights appear in the denominator of $r(s)$ as well. \aaa solves a relaxed linearized \ls problem instead.  
For a sampling point $\sh_i$ in the set $\bolds$, \aaa uses the linearization
\begin{equation} \label{linearaaa}
\hh_i - r(\sh_i)
= 
\dfrac{1}{q(\sh_i)} \left( \hh_i q (\sh_i) - p (\sh_i) \right) 
\rightsquigarrow 
\hh_i q (\sh_i) - p (\sh_i),  
\end{equation}
leading to the linearized \ls problem 
\begin{align} \label{aaals}
\min_{w_1,\ldots,w_{k}} \sum_{i=1}^{N_s-n} 
\mid \hh_i q(\sh_i) - p(\sh_i) \mid^2.
\end{align}

\aaa is an iterative algorithm and builds the partitioning \cref{xixihat} using a greedy search.  Assume in  step $n$, 
\aaa has the rational approximant $r(s)$ as in \cref{eq:bary_rep} corresponding to the  partitioning \cref{xixihat} where the weights $\{w_k\}$ are selected by solving \cref{aaals}. \aaa updates \cref{xixihat} via a greedy search by finding $\sh_i \in \boldsh$ for which the error $\mid r(\sh_i) - \hh_i\mid$  is the largest. This sampling point is, then, added to the interpolation set $\bolds$, the barycentric  rational approximant $r(s)$
in \cref{eq:bary_rep} is updated accordingly (it has  one higher degree now), and the new weights are computed, as before, solving the  linearized \ls problem. The procedure is repeated until either a desired order or an error tolerance is obtained. 
For further details, we refer the reader to the original source \cite{NST18}.  The \aaa algorithm proved very flexible and effective, and has been employed in various applications such as  rational approximation over disconnected domains \cite{NST18}, 
solving nonlinear eigenvalue problems \cite{lietaert2018automatic}, modeling of parametrized dynamics \cite{CS20}, and approximation of matrix-valued functions \cite{GG20}.

\section{Barycentric representations for 
	\lqo systems} \label{sec:lqobary}

To develop interpolating barycentric forms for $H_1(s)$
and $H_2(s,z)$, we first need to specify the data corresponding to the underlying
the \lqo system $\Sigma_{\lqo}$. 
The first transfer function $H_1(s)$ of 
$\Sigma_{\lqo}$ 
is a regular single-variable rational function and, as in \Cref{sec:aaa}, we sample 
$H_1(s)$ at  distinct points 
$\{s_1,\ldots,s_{N_s}\}$
to obtain the data
set
\begin{equation}  \label{H1fulldata}
\{H_1(s_i)\} \in \mathbb{C}~~\mbox{where}~~s_i \in \mathbb{C}\quad\mbox{for}~~ i=1,2,\ldots,N_s.
\end{equation}
The second transfer function $H_2(s,z)$,
on the other hand, is a function of two-variables. Therefore, in agreement with the data \cref{H1fulldata}, we will  sample 
$H_2(s,z)$ at the corresponding rectangular grid:
for $i,j=1,2,\ldots,N_s$,
\begin{equation}  \label{H2fulldata}
\{H_2(s_i,s_j)\} \in \mathbb{C}~~\mbox{where}~~s_i,s_j \in \mathbb{C}.
\end{equation}

Partition the full set of sampling points into two disjoint sets
{\small \begin{equation}  \label{spart}
	\{ s_1, \dots, s_{N_s} \} =
	\{~\s_1, \dots, \s_{n} ~\}  \cup  
	\{~\sh_1, \dots, \sh_{N_s-n}~\}
	= \bolds~ \cup~ \boldsh
	\end{equation}}
and define the sampled values (measurements):
\begin{equation} \label{data_H1}
h_i \eqdef H_1(\s_i)~~~\mbox{for}~~i=1,2,\ldots,n
\end{equation}
and 
\begin{align} \label{data_H2}
h_{i,j} \eqdef H_2(\s_i,\s_j) ~~~\mbox{for}~~i,j=1,2,\ldots,n.
\end{align}
Then, the goal is to a construct a \emph{data-driven}  \lqo system directly from its samples without access access to  internal dynamics of $\Sigma_{\lqo}$. The
partition \cref{spart} and the error measure used in approximating the data 
will be clarified later. First we will show how the data  in \cref{H1fulldata} and \cref{H2fulldata} can be used to develop barycentric-like representations corresponding to a reduced \lqo system. We will use the notation $r_1(s)$ to denote the rational approximation to $H_1(s)$ and  $r_2(s,z)$ to  
$H_2(s,z)$.

\begin{proposition} \label{r1thm}
	Given the $H_1(s)$ samples in \cref{H1fulldata}, 
	pick the nonzero weights $\{w_1,w_2,\ldots,w_n\}$.
	Then, the barycentric rational function
	\begin{align} \label{H1_bary}
	\hspace{-3ex} r_1(s) = \dfrac{n_1(s)}{d_1(s)} = 
	{\displaystyle \sum_{k=1}^n \frac{ w_k h_k }{s - \s_k }} \Bigg/\left( {1+\displaystyle \sum_{k=1}^n \frac{w_k }{s- \s_k}}\right)
	\end{align}
	interpolates the data in \cref{data_H1}.
	Let $\bfe\in \mathbb{C}^n$ denote the vector of ones.  Define the matrices
	\begin{align}  \label{r1ss}
	\begin{split} 
	\bhb ~&=~ \left[ \begin{array}{cccc} w_1 & w_2 & \ldots & w_n \end{array} \right]^T \in \mathbb{C}^n, \\
	\mathbf{\Xi} ~&=~ {\mathsf{diag}}(\s_1,\ldots,\s_n) \in
	\mathbb{C}^{n\times n} \\
	\bhA ~&=~ \mathbf{\Xi}  -  \bhb \bfe^T\in
	\mathbb{C}^{n\times n},~\mbox{and}\\ 
	\bhc^T ~&=~ \left[ \begin{array}{cccc} h_1 & h_2 & \ldots & h_n \end{array} \right]\in
	\mathbb{C}^{n}.
	\end{split}
	\end{align}
	Then, $r_1(s)$ has the state-space form
	\begin{equation}  \label{r1tf}
	r_1(s)= \bhc^T(s\hat{\bI}-\bhA)^{-1}\bhb,
	\end{equation}
	where $\hat{\bI}$ is the identity matrix of dimension
	$n\times n$.
\end{proposition}
\begin{proof}
	The fact that $r_1(s)$ is an interpolating rational function for 
	the data \cref{data_H1} is just a restatement of \cref{singleinterp} for completeness. To prove \cref{r1tf}, 
	we will use the \textit{Sherman-Morrison} formula \cite{GoluVanl2013}: Let $\bM  \in \mathbb{C}^{n \times n}$ be an invertible and $\bu,\bv \in \mathbb{C}^n$ be such that 
	$1 + \bv^* \bM^{-1} \bu \neq 0.$ Then,  
	\begin{align} \label{smw}
	(\bM + \bu \bv^*) ^ {-1} = \bM^{-1} - \frac{\bM^{-1} \bu \bv^* \bM^{-1}}{ 1+\bv^* \bM^{-1} \bu}.
	\end{align}
	From \cref{r1ss} and \cref{r1tf}, we have 
	\begin{equation} \label{r1new}
	r_1(s) 
	= \bhc^T (s \hat\bI - \bhA)^{-1} \bhb =  \bhc^T [(s \bI - \mathbf{\Xi})+\bhb\bfe^T]^{-1} \bhb.
	\end{equation}
	To simplify the notation, let
	$\hat{\bPhi}_{s} = s \bI_n - \mathbf{\Xi}$. Then, applying the Sherman-Morrison formula to the middle term
	in \cref{r1new}
	with $\bM =\hat{\bPhi}_{s}$, $\bu = \bhb$, and $\bv = \bfe$, we obtain
	\begin{align} 
	r_1(s) &=  \bhc^T \left(\hat{\bPhi}_{s}+\bhb \bfe^T\right)^{-1} \bhb \nonumber \\
	&= 
	\bhc^T \left( \hat{\bPhi}_s^{-1} - \frac{\hat{\bPhi}_{s}^{-1} \bhb \bfe^T \hat{\bPhi}_s^{-1}}{ 1+\bfe^T \hat{\bPhi}_{s}^{-1} \bhb} \right) \bhb \nonumber \\
	&= \bhc^T \left( \hat{\bPhi}_{s}^{-1} \bhb - \frac{  \hat{\bPhi}_{s}^{-1} \bhb \cdot \bfe^T \hat{\bPhi}_{s}^{-1} \bhb}{1+\bfe^T \hat{\bPhi}_{s}^{-1} \bhb} \right) \nonumber \\
	&= \bhc^T\frac{  \hat{\bPhi}_{s}^{-1} \bhb}{1+\bfe^T  \hat{\bPhi}_{s}^{-1} \bhb}. \label{r1final}
	\end{align}
	Since $\mathbf{\Xi}$ 
	is diagonal,  
	$$\hat{\bPhi}_{s}^{-1} = (s \bI - \mathbf{\Xi})^{-1} =  \mathsf{diag}(\left[ \begin{array}{ccc} (s - \s_1)^{-1} & \ldots & (s - \s_n)^{-1}  \end{array} \right]).$$
	Then, using the definitions of $\bhb$ and $\bhc$
	in \cref{r1ss},  we obtain
	\begin{equation}  \label{phihatsb}
	\bhc^T\hat{\bPhi}_{s}^{-1} \bhb = \displaystyle \sum_{k=1}^n \frac{ w_k h_k }{s - \s_k } 
	~~\mbox{and}~~
	\bfe^T \hat{\bPhi}_{s}^{-1} \bhb = \displaystyle \sum_{k=1}^n \frac{ w_k}{s - \s_k }. 
	\end{equation}
	Substituting these last two equalities into \cref{r1final} 
	yields \cref{r1tf}.
\end{proof}
We note that state-space realizations for  rational functions are unique up to a similarity transformations. For other equivalent state-space representations of a barycentric form, we refer the reader to, e.g., \cite{ALI17,lietaert2018automatic}.

Given the samples of $H_1(s)$ ( the data in \cref{data_H1}) of the \lqo system \cref{eq:def_LQO}, \Cref{r1thm} constructs the \emph{linear} part of the data-driven \lqo model, directly from these samples. What we need to achieve next is to use the  $H_2(s,z)$ samples (data in \cref{data_H2}) to construct  a two-variable rational function $r_2(s,z)$ in a barycentric-like form corresponding to the quadratic part of the data-driven \lqo model. However, $r_2(s,z)$ cannot be constructed independently from $r_1(s)$. Once  $r_2(s,z)$ is constructed, we should be able to interpret $r_1(s)$ and $r_2(s,z)$ as the linear and quadratic transfer functions of a single \lqo system. This is the precise reason that we cannot simply view $r_2(s,z)$ as an independent two-variable rational function and  use the classical multivariate barycentric form \cite{An12two,AntBG20}. Therefore,
$r_2(s,z)$ needs to have the form
$$
r_2(s,z) = \hat{\bK} \big{[} (s\hat{\bI}-\bhA)^{-1} \bhb \otimes (z\hat{\bI}-\bhA) \bhb \big{]},
$$
where  $\bhA$ and $\bhb$ are the same matrices 
from \cref{r1ss} used in modeling $r_1(s)$ and 
$\hat{\bK} \in \IC^{1 \times n^2}$ is the (quadratic) free variable that will incorporate to model the new data \cref{data_H2}. The next result achieves this goal.

\begin{theorem}  \label{r2thm}
	Assume the set-up in \Cref{r1thm}.
	Further assume that the $H_2(s,z)$ samples in \cref{H2fulldata} are given. 
	Define, the two-variable  function $r_2(s,z)$ in a barycentric-like form:
	\small
	\begin{equation}\label{eq:H2_bary}
	r_2(s,z) = \frac{\displaystyle \sum_{k=1}^n \sum_{\ell=1}^n \frac{ h_{k,\ell}w_k w_\ell}{(s-\s_k)(z-\s_\ell)}}{1+\displaystyle \sum_{k=1}^n \frac{w_k }{s- \s_k} + \displaystyle \sum_{\ell=1}^n \frac{w_\ell }{z- \s_\ell} +\displaystyle \sum_{k=1}^n \sum_{\ell=1}^n \frac{ w_k w_\ell}{(s-\s_k)(z-\s_\ell)}}.
	\end{equation}
	\normalsize
	Then, $r_2(s,z)$ interpolates the data \cref{data_H2}, i.e.,
	\begin{equation} \label{h2intpl}
	r_2(\s_i,\s_j)  = H_2(\s_i,\s_j)~~\mbox{for} ~~i,j = 1,\ldots,n.
	\end{equation}
	Define $\hat{\bM} \in 
	\mathbb{C}^{n \times n} $ and $\hat{\bK} \in 
	\mathbb{C}^{1 \times n^2}$ using 
	\begin{align}
	[\hat{\bM}]_{i,j} &= h_{i,j}~~\mbox{for}~~i,j=1,2,\ldots,n \\
	~~\mbox{and}~~\hat{\bK} &= [\mathsf{vec}(\hat{\bM)}]^T.  \label{hatKdef}
	\end{align}
	Then,  $r_2(s,z)$ has the state-space form
	\begin{equation} \label{r2tf}
	r_2(s,z) = \hat{\bK} \big{[} (s\hat{\bI}-\bhA)^{-1} \bhb \otimes (z\hat{\bI}-\bhA)^{-1} \bhb \big{]}.
	\end{equation}
\end{theorem}
\begin{proof}
	To prove the interpolation property
	\cref{h2intpl} of the barycentric representation  \cref{eq:H2_bary},
	inspired by the linear case, 
	we start by introducing various polynomials in one or  two variables:
	\begin{align}
	\hspace{-1ex}
	\begin{split}
	p(s) &= \prod_{k=1}^n (s-\xi_k), \ \ p(z) = \prod_{\ell=1}^n (z-\xi_{\ell}),  \\
	p_i(s) &= \prod_{k=1,k \neq i}^n (s-\xi_k), \ \ p_j(z) = \prod_{\ell=1,\ell \neq j}^n (z-\xi_{\ell}),  \\
	P(s,z) &= \prod_{k=1}^n \prod_{\ell=1}^n (s-\xi_k)(z-\xi_\ell),~~\mbox{and}  \\
	P_{i,j}(s,z) &= \prod_{k=1, k \neq i}^n \prod_{\ell=1, \ell \neq j}^n (s-\xi_k) (z-\xi_\ell), 
	\end{split}
	\label{eq:definepij}
	\end{align}
	for  $i,j=1,\ldots,n$. Multiply both the numerator and denominator of $r_2(s,z)$ in (\ref{eq:H2_bary}) with
	$P(s,z)$ to obtain
	\begin{align}
	r_2(s,z) &= \frac{n_2(s,z)}{d_2(s,z)},  \label{r2n2d2}
	\end{align}
	\text{with}
	\begin{equation}\label{formulae_n2_d2}
	\begin{split}
	n_2(s,z)  &= \displaystyle \sum_{k=1}^n \sum_{\ell=1}^n  h_{k,\ell}w_k w_\ell P_{k,\ell}(s,z),  \\
	d_2(s,z) &= P(s,z)+\displaystyle \sum_{k=1}^n w_k p_k(s) p(z) + \displaystyle \sum_{\ell=1}^n w_\ell p_\ell(z) p(s)  \\
	&+\displaystyle \sum_{k=1}^n \sum_{\ell=1}^n w_k w_\ell P_{k,\ell}(s,z).
	\end{split}
	\end{equation}
	Then, evaluate $r_2(s,z)$ at $s = \xi_i$ and $z = \xi_j$ to obtain
	\begin{align*}
	\hspace{-5mm}   r_2(\xi_i,\xi_j) =  \frac{n_2(\xi_i,\xi_j)}{d_2(\xi_i,\xi_j)}
	= \frac{h_{i,j}w_i w_j P_{i,j}(\xi_i,\xi_j)}{w_i w_j P_{i,j}(\xi_i,\xi_j)} = h_{i,j}.
	\end{align*}
	To prove  \cref{r2tf}, we first note that 
	\begin{align*} 
	\begin{split}
	r_2(s,z) &=   \hat{\bK} \left[ (s\hat{\bI}-\bhA)^{-1} \bhb \otimes (z\hat{\bI}-\bhA)^{-1} \bhb \right] \\
	&= \hat\bK \left[  \frac{  \hat{\bPhi}_{s}^{-1} \bhb}{1+\bfe^T  \hat{\bPhi}_{s}^{-1} \bhb} \otimes \frac{  \hat{\bPhi}_{z}^{-1} \bhb}{1+\bfe^T  \hat{\bPhi}_{z}^{-1} \bhb} \right], 
	\end{split}
	\end{align*}
	where we used the fact 
	$$
	(s\hat\bI - \bhA)^{-1} \bhb = \left(\hat{\bPhi}_{s}+\bhb \bfe^T\right)^{-1}\bhb =
	\frac{  \hat{\bPhi}_{s}^{-1} \bhb}{1+\bfe^T  \hat{\bPhi}_{s}^{-1} \bhb},
	$$
	as shown in deriving \cref{r1final}. 
	Since $\hat{\bPhi}_{s}$ diagonal,  we have
	{\small 
		\begin{align*}
		\begin{split}
		r_2(s,z) =
		\frac{\hat\bK }{\big{(} 1+\bfe^T  \hat{\bPhi}_{s} \bhb \big{)}\big{(} 1+\bfe^T  \hat{\bPhi}_{z} \bhb \big{)}}  \left[ \begin{matrix}
		\frac{w_1}{s-\s_1} \\
		\vdots \\ \frac{w_n}{s-\s_n}  
		\end{matrix} \right] \otimes \left[ \begin{matrix}
		\frac{w_1}{z-\s_1} \\
		\vdots \\ \frac{w_n}{z-\s_n}  
		\end{matrix} \right].  
		\end{split}
		\end{align*}
	}
	Then, using the definition of $\hat\bK$ 
	in \cref{hatKdef} together with the second 
	formula in \cref{phihatsb}, we obtain
	{\small 
		\begin{align} 
		&r_2(s,z) =  \frac{ \sum_{k=1}^n \sum_{\ell=1}^n \frac{h_{k,\ell}  w_k w_\ell}{(s-\s_k)(z-\s_\ell)}}{\Big{(}1+ \sum_{k=1}^n \frac{w_k }{s- \s_k}\Big{)}\Big{(}1+ \sum_{\ell=1}^n \frac{w_\ell }{z- \s_\ell}\Big{)}} \label{r2denom} \\[2mm]
		& = \frac{\sum_{k=1}^n \sum_{\ell=1}^n \frac{h_{k,\ell}  w_k w_\ell}{(s-\s_k)(z-\s_\ell)}}{1+\displaystyle \sum_{k=1}^n \frac{w_k }{s- \s_k} +  \sum_{\ell=1}^n \frac{w_\ell }{z- \s_\ell} + \sum_{k=1}^n \sum_{\ell=1}^n \frac{  w_k w_\ell}{(s-\s_k)(z-\s_\ell)}}, \nonumber
		\end{align}
	}
	which concludes the proof. 
\end{proof}
The next result follows from \Cref{r1thm,r2thm}.
\begin{corollary} \label{thisislqo}
	Assume the set-up in \Cref{r1thm,r2thm}. Then, interpolating rational functions $r_1(s)$ and $r_2(s,z)$ together correspond to an interpolatory  \lqo model
	\begin{align}\label{redLQO}
	\widehat{\Sigma}_{\lqo}: \begin{cases}  \dot{\hat{\bx}}(t)=\bhA\bx(t)+\bhb u(t),\\
	\hspace{0.5mm} \hat{y}(t)=\bhc^T\hat{\bx}(t)
	+{\hat{\bK} \big{[} \hat{\bx}(t) \otimes \hat{\bx}(t) \big{]}}.
	\end{cases}
	\end{align}
	In others words, the first (linear) transfer function 
	of $\widehat{\Sigma}_{\lqo}$
	is $r_1(s)$ and its second transfer function is $r_2(s,z)$. 
\end{corollary}

Recall the partitioning of the sampling points in
\eqref{spart}. In  \Cref{r2thm}, we have shown that $r_2(s,z)$ interpolates $H_2(s,z)$ over the sampling set $\bolds \times \bolds$. What is the value of $r_2(s,z)$  over the \emph{mixed} sampling sets  $\bolds \times \boldsh$ and  $\boldsh \times \bolds$? Even though we do not enforce interpolation over these sets, in \Cref{sec:prop_meth} we will need a closed-form expression for the value of $r_2(s,z)$
over $\bolds \times \boldsh$ and  $\boldsh \times \bolds$. The next lemma establishes these results. 
\begin{lemma}\label{Lemma2}
	Let $r_2(s,z)$ be as defined in
	\cref{eq:H2_bary}
	corresponding to the sampling points in
	\cref{spart} and the data in \cref{H2fulldata}. Then,
	{\small
		\begin{align}\label{eval_r2_s_sh}
		\hspace{0mm} r_2(\s_i,\sh_j) &= \frac{\displaystyle  \sum_{\ell=1}^n   \frac{w_\ell h_{i,\ell} }{\sh_j-\s_\ell}}{1+\displaystyle \sum_{\ell=1}^n  \frac{ w_\ell}{\sh_j-\s_\ell}} \ \ ~\mbox{and} ~\ \  r_2(\sh_j,\s_i) = \frac{\displaystyle  \sum_{k=1}^n   \frac{w_k h_{k,i} }{\sh_j-\s_k}}{1+\displaystyle \sum_{k=1}^n  \frac{ w_k}{\sh_j-\s_k}},
		\end{align}
	}
	for $i = 1,\ldots, n$ and $j = 1, \ldots, N_s-n$.
\end{lemma}

\begin{proof}
	Proof is given in \Cref{sec:appendix}.
\end{proof}
{
	It is important to note that the numerators and denominators of $r_2(\s_i,\sh_j)$
	and $r_2(\sh_j,\s_i)$ in \cref{eval_r2_s_sh} are \emph{linear} in the weights $w_\ell$. This is in contrast to 
	the general form of $r_2(s,z)$  in \cref{r2denom} where both the numerator and denominator are quadratic in $w_\ell$ when evaluated over $\boldsh \times \bolds$.}

\section{Proposed framework for data-driven modeling of \lqo systems}
\label{sec:prop_meth}

\Cref{sec:lqobary} established the necessary ingredients to extend \aaa to \lqo systems. Given the measurements 
\cref{H1fulldata} and \cref{H2fulldata}, 
\Cref{r1thm,r2thm} show how to construct the barycentric forms $r_1(s)$ and $r_2(s,z)$ interpolating this data in accordance 
with the partitioning \cref{spart}. Furthermore, \Cref{thisislqo} states that $r_1(s)$ and $r_2(s,z)$ together correspond to an interpolatory \lqo system. Based on these results, in this section we will fully develop the \aaa framework for \lqo systems. The resulting  algorithm will be denoted by \lqoaaa. 

\lqoaaa will be an iterative algorithm, adding one degree of freedom to the current data-driven \lqo model in every iteration step. In the $n$th step,
$r_1^{(n)}(s)$ and $r_2^{(n)}(s,z)$ will correspond to a data-driven order-$n$ \lqo model for the  partitioning of the sampling points in
\cref{spart}. First, for this \emph{current} partitioning, in \Cref{sec:computew},
we introduce a \ls error measure that will be used to choose the barycentric weights $\{w_k\}$ appearing  in the definitions of $r_1^{(n)}(s)$ and $r_2^{(n)}(s,z)$ in
\cref{H1_bary} and \cref{eq:H2_bary}. 
Then, in \Cref{sec:greedy} 
we establish a greedy search procedure for updating the  partitioning  \cref{spart}. The algorithm will then continue with the \ls minimization for the updated partitioning at the $(n+1)$th step to construct 
$r_1^{(n+1)}(s)$ and $r_2^{(n+1)}(s,z)$.
\lqoaaa will terminate after a desired error criterion is met or a maximum allowed order is achieved as explained in \Cref{sec:lqoaaa}.

Even though \Cref{sec:computew} investigates the 
\ls problem in the $n$the of \lqoaaa, to simplify the notation,
we drop the superscript and use $r_1(s)$ and $r_2(s,z)$ instead.  However, they should be understood as the approximants in the $n$th step. We will reintroduce the superscript in \Cref{sec:greedy}.

\subsection{A combined \ls measure for computing the barycentric weights for the current partition} \label{sec:computew}

For the full \lqo data
\cref{H1fulldata} and \cref{H2fulldata}, 
we recall (and repeat) the  partitioning of the sampling points as in \cref{spart}:
{\small 
	\begin{equation}
	\{ s_1, \dots, s_{N_s} \} =
	\{~\s_1, \dots, \s_{n} ~\}  \cup  
	\{~\sh_1, \dots, \sh_{N_s-n}~\} = \bolds ~\cup~ \boldsh. 
	\label{spartrepeat}
	\end{equation}
}
Then, $r_1(s)$ interpolates $H_1(s)$ over $\bolds$ (i.e., it interpolates the data \cref{data_H1}) and $r_2(s,z)$ interpolates $H_2(s,z)$ over  $\bolds \times \bolds$ (i.e., it interpolates the data \cref{data_H2}). Also recall that 
together, $r_1(s)$ and $r_2(s,z)$ correspond to a 
\lqo  system.
The only remaining degrees of freedom in defining $r_1(s)$ and $r_2(s,z)$, and thus the corresponding \lqo  system are the barycentric weights $\{w_1,\ldots,w_n\}$. We will choose those weights to minimize an appropriate error 
measure in the \emph{uninterpolated} data corresponding to the sampling points $\boldsh$. We first introduce the notation for  these uninterpolated  values\footnote{Since the evaluation of the uninterpolated $H_2(s,z)$ values 
	occur over three different sets, namely 
	$\bolds \times \boldsh$, 
	$\boldsh \times \bolds$, and $\boldsh \times \boldsh$, we use a superscript to distinguish them. Recall that the interpolated values $h_{i,j} = H_2(\s_i,\s_j)$ are over $\bolds \times \bolds$ only and thus the superscript notation is avoided for 
	$h_{i,j}$.}:
{
	\begin{align} \label{data_H1hat}
	\hspace{-8mm}  \hh_i \eqdef H_1(\sh_i)\phantom{,\sh_j}~~\mbox{for}~i=1,2,\ldots,N_s-n,\hspace{17mm} \\
	\hspace{-5mm} \hij  \eqdef H_2(\sh_i,\sh_j) \label{data_H2hat} ~~\mbox{for}~i,j=1,2,\ldots,N_s-n,\hspace{14mm}\\
	\label{data_H2_nonhat_hat}
	\hspace{+2mm} \gij  \eqdef H_2(\s_i,\sh_j)~~\mbox{for}~i = 1,\ldots,n, ~j=1,\ldots,N_s-n, \\
	\label{data_H2_hat_nonhat}
	\hspace{+2mm} \fji \eqdef H_2(\sh_j,\s_i)~~\mbox{for}~j = 1,\ldots,N_s-n, ~i=1,\ldots,n.
	\end{align}}
We denote with $\bw \in \mathbb{C}^n$ the vector of weights to be determined, as
$$
\bw = \begin{bmatrix} w_1 & w_2 &\ldots &w_n \end{bmatrix}^T.
$$
A reasonable error measure to minimize is the \ls distance in the uninterpolated data, leading to the minimization problem 
\begin{align} \label{truenls}
\min\limits_{\bw \neq \mathbf{0}} \left(\J_1+ \J_2 + \J_3 + \J_4\right)
\end{align}
where 
\begin{align} 
\hspace{-2mm} \J_1 & = {\frac{1}{N_s-n}}\sum_{i=1}^{N_s-n} (r_1(\sh_i) -  \hh_i)^2,  \label{J1term} \\ 
\hspace{-2mm} \J_2 & ={\frac{1}{n(N_s-n)}} \sum_{i=1}^{n} \sum_{j=1}^{N_s-n} (r_2(\s_i,\sh_j) -\gij  )^2, \label{J2term}\\ 
\hspace{-2mm} \J_3  &= {\frac{1}{(N_s-n)n}} \sum_{i=1}^{N_s-n} \sum_{j=1}^{n} (r_2(\sh_i,\s_j) - \fji )^2,~\mbox{and}  \label{J3term} \\ 
\hspace{-2mm} \J_4  &= {\frac{1}{(N_s-n)^2}} \sum_{i=1}^{N_s-n} \sum_{j=1}^{N_s-n} (r_2(\sh_i,\sh_j) - \hij )^2.\label{J4term}
\end{align}
As in the original \aaa for linear dynamical systems, the \ls problem \cref{truenls}
is nonlinear in $\bw$ for \lqo systems. The formulation is  more complicated here due to the additional $r_2(s,z)$ term. To resolve this numerical difficulty, we will employ a strategy, similar to the lineraziation step in \cref{linearaaa}, and solve a  
relaxed optimization problem. However, 
the resulting \ls problem in our case will still be nonlinear, yet much easier to solve than \cref{truenls}. In the end, we will tackle the original nonlinear \ls problem \cref{truenls} by solving a sequence of quadratic  \ls problems. 
We note that in \cref{J1term}-\cref{J4term}, we scale every error term $\J_i$ with the number of data points in it. 
\subsubsection{Quadraticized \ls problem in Step $n$}
In this section, we show how to relax the each term, $\J_i$, in the nonlinear \ls problem \cref{truenls}. The resulting problem will then constitute a crucial component in the proposed iterative algorithm (\Cref{sec:lqoaaa}).
\paragraph{\textbf{Linearizing $\J_1${\textrm :}}} \label{linj1}
Note that the $i$th term of $\J_1$ in \cref{J1term}, namely 
$r_1(\sh_i) -  \hh_i$, is the same as the term in \cref{linearaaa} appearing in \aaa. This is natural since 
$r_1(s)$ corresponds to the linear part of the \lqo system. Therefore, we can linearize $\J_1$ similar to \cref{linearaaa}. 
Write $r_1(s)$ as $r_1(s) = n_1(s)/d_1(s)$, as defined in \eqref{H1_bary}. Then,
the $i$th term in \cref{J1term} is linearized as 
\begin{equation} \label{linearJ1}
r_1(\sh_i) - \hh_i
= 
\dfrac{1}{d_1(\sh_i)} \left(n_1(\sh_i) - \hh_i d_1 (\sh_i)  \right) 
\rightsquigarrow 
n_1 (\sh_i)-   \hh_i d_1 (\sh_i). 
\end{equation}
Substituting $n_1(s)$ and $d_1(s)$ from the definition of $r_1(s)$  
in \cref{H1_bary} into \cref{linearJ1}, one obtains 
\begin{align}
\nonumber
n (\sh_i)-   \hh_i d (\sh_i) &=
\sum_{k=1}^n \frac{ w_k h_k }{\sh_i - \s_k } - 
\hh_i \left(
1+ \sum_{k=1}^n \frac{w_k }{\sh_i- \s_k} \right)
\\
&=\sum_{k=1}^n \frac{w_k(h_k-\hh_i)}{\sh_i- \s_k} - \hh_i. 
\end{align}
For a matrix $\bX$, let $\left(\bX\right)_{ij}$ denote its $(ij)$th entry. Similarly, for a vector $\bx$, let $\left(\bx\right)_{i}$ denote its $i$th entry. Then, 
define the Loewner matrix $\IL \in \mathbb{C}^{(N_s-n)\times n}$ with
\begin{equation}
\left(\IL\right)_{ik} = \frac{\hh_i-h_k}{\sh_i- \s_k},~~\mbox{for}~~
i=1,\ldots,N_s,~k=1,\ldots,n,
\end{equation}
and  the vector $\bhh \in \mathbb{C}^{N_s-n}$ with $\left(\bhh\right)_i = \hh_i$.  Then, 
$$
\sum_{i=1}^{N_s-n} \left( n (\sh_i)-   \hh_i d (\sh_i) \right)^2
= \Vert \IL \bw +  \bhh \Vert_2^2.
$$
Therefore, the $\J_1$ term in \cref{J1term} will be relaxed to
\begin{align}  \label{J1relax}
\J_1  \longrightsquigarrow  {\frac{1}{N_s-n}} \Vert \IL \bw +  \bhh \Vert_2^2.
\end{align}

\paragraph{\textbf{Linearizing $\J_2$ and $\J_3${\textrm :}}}  \label{linj2j3}

Now we extend the linearization strategy used in $\J_1$, which only involved the single-variable function $r_1(s)$,  to the error terms $\J_2$ and $\J_3$, which involve $r_2(s,z)$.  The closed-form expressions for $ r_2(\s_i,\sh_j)$ and $r_2(\sh_j,\s_i)$ we derived in \Cref{Lemma2} will prove fundamental in achieving these goals. 

We start with $\J_2$. Write $r_2(s,z) = n_2(s,z)/d_2(s,z)$ as  in \cref{r2n2d2}.
Then, the linearizing the $(ij)th$ term in \eqref{J2term} means 
\begin{align}  
\hspace*{-1ex}      r_2(\s_i,\sh_j) - \gij 
& = 
\dfrac{1}{d_2(\s_i,\sh_j)} \left(n_2(\s_i,\sh_j) - \gij d_1 (\s_i,\sh_j)  \right) \nonumber \\
&\rightsquigarrow 
n_2 (\s_i,\sh_j)-   \gij d_2 (\s_i,\sh_j). \label{linearJ2}
\end{align}
We substitute $n_2(\s_i,\sh_j)$ and $d_2(\s_i,\sh_j)$ from   \cref{eval_r2_s_sh} into \cref{linearJ2} to obtain
\begin{align}
n_2 (\s_i,\sh_j) & -   \gij d_2 (\s_i,\sh_j) \nonumber \\
& = \sum_{\ell=1}^n   \frac{w_\ell h_{i,\ell} }{\sh_j-\s_\ell} -  \gij
\left({1+ \sum_{\ell=1}^n  \frac{ w_\ell}{\sh_j-\s_\ell}}\right) \nonumber \\
& = -\left( \sum_{\ell=1}^n \frac{ w_\ell(\gij-h_{i,\ell})}{\sh_j- \s_\ell}  + \gij \right). \label{n2ghatd2}
\end{align}
Define the indexing variable $\alpha_{ij} = (i-1)(N_s-n)+j$ and let $\got \in \mathbb{C}^{n(N_s-n)}$ be the vector defined such that
\begin{equation} \label{ghat2}
\left(\got\right)_{\alpha_{ij}}= \gij~\mbox{for}~
1 \leqslant i \leqslant n ~\mbox{and}~ 1 \leqslant j\leqslant N_s-n.
\end{equation}
Define the  Loewner matrix $\IL^{(1,2)} \in \mathbb{C}^{n(N_s-n)\times n}$ with entries 
\begin{equation}
\left( \Lonetwo\right)_{\alpha_{ij}\ell} =
\frac{\gij-h_{i,\ell}}{\sh_j- \s_\ell},
\label{L12}
\end{equation}
for
$1 \leqslant i \leqslant n, \ 1 \leqslant j\leqslant N_s-n$, and 
$1 \leqslant \ell \leqslant n$. Then, 
using \eqref{L12} and \eqref{ghat2} in 
\eqref{n2ghatd2}, we obtain 
$$
\sum_{i=1}^n\sum_{j=1}^{N_s-n}
\left(n_2 (\s_i,\sh_j)  -   \gij d_2 (\s_i,\sh_j)\right)^2 =
\Vert \Lonetwo \bw +  \got \Vert_2^2,
$$
yielding the linearization of $\J_2$:
\begin{align} \label{J2relax}
\J_2  \longrightsquigarrow  {\frac{1}{(N_s-n)n}} \left\| \Lonetwo \bw +  \got \right\|_2^2.
\end{align}
Using similar arguments and the explicit formula for the expression $r_2(\sh_j,\s_i)$ in
\cref{eval_r2_s_sh}, the $\J_3$ term in \cref{J3term} is linearized to 
\begin{align} \label{J3relax}
\J_3  \longrightsquigarrow  \frac{1}{(N_s-n)n} \left\|\Ltwoone \bw +  \gto \right\|_2^2,
\end{align}
where the Loewner matrix $\IL^{(2,1)} \in \mathbb{C}^{n(N_s-n)\times n}$ and
the vector $\gto \in \mathbb{C}^{n(N_s-n)}$ are defined as 
\begin{align*}
\left(\Ltwoone \right)_{\gamma_{ji}k} = \frac{\fji-h_{k,i}}{\sh_j- \s_k}
\quad\mbox{and}\quad (\gto)_{\gamma_{ji}} = \fji,
\end{align*}
with $1 \leqslant j \leqslant N_s-n, \ 1 \leqslant i\leqslant n$, $1 \leqslant k \leqslant n$, and $\gamma_{ji} = (j-1)n+i$.


\paragraph{\textbf{Quadraticizing the $\J_4$ term{\textrm :}}} \label{linj4}

In this section we show how to relax the remaining term, $\J_4$, in the minimization problem \cref{truenls}. Note that this term includes $r_2(\sh_i,\sh_j)$; i.e., $r_2(s,z)$ evaluated over $\boldsh \times \boldsh$. As we stated earlier, unlike 
$r_2(\s_i,\sh_j)$ ($r_2(s,z)$  over $\bolds \times \boldsh$) or $r_2(\sh_i,\s_j)$ ($r_2(s,z)$  over $\bolds \times \boldsh$), the numerator and denominator of the quantity $r_2(\sh_i,\sh_j)$ is quadratic in the weights $w_\ell$. Therefore,  relaxing the $(ij)$th term in $\J_4$ via multiplying it out with its denominator, will not yield a linear term, but rather a quadratic, i.e., even the relaxed problem cannot be solved as a linear \ls problem.
This is what we establish next. 

Similar to \eqref{linearJ2}, relax the $(ij)$th term in
\eqref{J4term} using 
\begin{align}  
\hspace{-1ex}
r_2(\sh_i,\sh_j) - \hij 
& = 
\dfrac{1}{d_2(\sh_i,\sh_j)} \left(n_2(\sh_i,\sh_j) - \hij d_1 (\sh_i,\sh_j)  \right) \nonumber \\
&\rightsquigarrow 
n_2 (\sh_i,\sh_j)-   \hij d_2 (\sh_i,\sh_j). \label{linearJ4}
\end{align}
Using \eqref{r2denom}, we obtain
{\small
	\begin{equation}  \label{n2d2shatshat}
	r_2(\sh_i,\sh_j) = \frac{ \sum_{k=1}^n \sum_{\ell=1}^n \frac{ w_k w_\ell h_{k,\ell} }{(\sh_i - \s_k)(\sh_j - \s_\ell) }}{(1+ \sum_{k=1}^n \frac{w_k }{\sh_i- \s_k})(1+ \sum_{\ell=1}^n \frac{w_\ell }{\sh_j- \s_\ell})} 
	= \frac{ n_2 (\sh_i,\sh_j)}{d_2 (\sh_i,\sh_j)}.
	\end{equation}
}
Inserting $n_2 (\sh_i,\sh_j)$ and $d_2 (\sh_i,\sh_j)$
from \eqref{n2d2shatshat} into \eqref{linearJ4} and re-arranging 
the terms yields
\begin{align}
n_2 (\sh_i,&\sh_j)-   \gij d_2 (\s_i,\sh_j) \nonumber \\
= &-\left( \sum_{k=1}^n \sum_{\ell=1}^n \frac{w_k w_\ell(\hij-h_{k,\ell})}{(\sh_i- \s_k)(\sh_j- \s_\ell)} \right. \nonumber \\
& \quad \quad \left. +  \sum_{k=1}^n \frac{w_k \hij}{\sh_i- \s_k} + \sum_{\ell=1}^n \frac{w_\ell \hij}{\sh_j- \s_\ell} - \hij \right). \label{n2hhatd2}
\end{align}
Note that the  expression in \eqref{n2hhatd2} is quadratic in $w_k$, as anticipated. 

As we did for the $\J_1$, $\J_2$ and $\J_3$, to express the resulting expression more compactly in matrix form,
we introduce the (2D) Loewner matrix
$\Ltwod  \in \mathbb{C}^{(N_s-n)^2 \times n^2}$ as 
\begin{equation} \label{L2d}
(\Ltwod)_{\alpha_{ij}  \beta_{k\ell}} =  \frac{\hij-h_{k,\ell}}{(\sh_i- \s_k)(\sh_j - \s_\ell)},
\end{equation}
where $\alpha_{ij} = (i-1)(N_s-n)+j$  and $\beta_{k\ell} = (k-1)n+\ell$ with 
$i,j \in \{1,2,\ldots, N_s-n\}$ and $k,\ell \in \{1,2,\ldots,n\}$.
Then, the $\alpha_{ij}$th entry of the vector $\Ltwod (\bw \otimes \bw) \in \mathbb{C}^{(N_s-n)^2}$ is  
\begin{equation}
\Big{(} \Ltwod (\bw \otimes \bw) \Big{)}_ {\alpha_{ij}} = -\sum_{k=1}^n \sum_{\ell=1}^n \frac{w_k w_\ell(h_{k,\ell}-\hij)}{(\sh_i- \s_k)(\sh_j- \s_\ell)},
\end{equation}
thus recovering the first sum in \eqref{n2hhatd2}.
Next, introduce the matrices $\bU_1, \bU_2 \in \mathbb{C}^{(N_s-n)^2 \times n}$ such that 
for $1\leqslant k,\ell \leqslant n$,
\begin{align}\label{def_U1_U2}
(\bU_1)_{\alpha_{ij}k} = \frac{\hij}{\sh_i- \s_k} \quad \mbox{and}\quad (\bU_2)_{\alpha_{ij}\ell} = \frac{\hij}{\sh_j- \s_\ell}.
\end{align}
Using $\bU_1$ and $\bU_2$ in \cref{def_U1_U2}, the last two sums in \eqref{n2hhatd2} can be compactly written as
\begin{align}
\sum_{k=1}^n \frac{w_k \hij}{\sh_i- \s_k} = \bU_1 \bw \quad \mbox{and}\quad  \sum_{\ell=1}^n \frac{w_\ell \hij}{\sh_j- \s_\ell}  = \bU_2 \bw.
\end{align}
Define $\bU = \bU_1+\bU_2$. Then using \cref{def_U1_U2}, we write
\begin{equation}\label{def_U}
(\bU)_{\alpha_{ij}k} =(\bU_1)_{\alpha_{ij}k} + (\bU_2)_{\alpha_{ij}k} = \frac{\hij(\sh_i+\sh_j-2\s_k)}{(\sh_i- \s_k)(\sh_j- \s_k)}.
\end{equation}
Insert \eqref{L2d} and \eqref{def_U} 
into \eqref{n2d2shatshat}  obtain 
\begin{align} \nonumber
\sum_{i=1}^{N_s-n} \sum_{j=1}^{N_s-n}
&(n_2 (\sh_i,\sh_j)  -   \hij d_2 (\s_i,\sh_j))^2  \\ & =
\left \| \Ltwod  ( \bw \otimes \bw) + \bU \bw + \gtt 
\right\|_2^2, \label{J4termnew}
\end{align}
where 
$ \gtt \in \mathbb{C}^{(N_s-n)^2}$ is the vector defined as 
\begin{equation}
(\gtt)_{\alpha_{ij}} = \hij,
\end{equation}
with $\alpha_{ij} = (i-1)(N_s-n)+j$ as before and $1 \leqslant i,j\leqslant N_s-n$.
The expression \eqref{J4termnew} yields the final relaxation of $\J_4$:
\begin{align}   \label{J4relax}
\J_4  \longrightsquigarrow  {\frac{1}{(N_s-n)^2}} \left\| \Ltwod ( \bw \otimes \bw) + \bU \bw + \gtt \right\|_2^2 .
\end{align}

\subsubsection{Solving the optimization problem in Step $n$}
Combining the relaxations $\J_1$, $\J_2$, $\J_3$, and $\J_4$ as given in \eqref{J1relax}, \eqref{J2relax}, \eqref{J3relax}, and \eqref{J4relax},
at the $n$th step  of the algorithm,
we need to solve the quadraticized minimization problem
\begin{align}
\hspace{-15mm}\min\limits_{\bw} \left\{ \rho_1 \left\| \IL \bw + 
\bhh \right\|_2^2 \right. \hspace{40mm} \nonumber \\  \hspace{5mm} + \rho_2\left( \left\|\Lonetwo \bw +  \got \right\|_2^2
+   \left\| \Ltwoone \bw +  \gto \right\|_2^2 \right) + \label{min_prob_1D_2D} \\
\left. \rho_3 \left\| \Ltwod ( \bw \otimes \bw) + \bU \bw + \gtt
\right\|_2^2 \right\}, \nonumber
\end{align}
where 
\begin{equation}
\rho_1 = \frac{1}{N_s-n},~\rho_2 =\frac{1}{(N_s-n)n},~\mbox{and}~\rho_3= \frac{1}{(N_s-n)^2}.
\end{equation}
Note that due to the last term, the optimization problem \eqref{min_prob_1D_2D}
is no longer a linear \ls problem, nevertheless  can be solved efficiently. One can explicitly compute the gradient (and Hessian) of the cost function and can apply a well-established (quasi)-Newton formulation \cite{nocedal2006numerical}. If we were to have a one-step algorithm whose solution is given by 
\eqref{min_prob_1D_2D}, one would employ these techniques.  However, note that solving \eqref{min_prob_1D_2D} is only  one step of our proposed iterative algorithm. Hence, as the iteration continues (and $n$ increases) the vector $\bw$ (and the data-partition) will be updated and the new optimization problem with a larger-dimension needs to be solved. Therefore, we will approximately solve
\eqref{min_prob_1D_2D} in every step.

One can obtain an approximate solution to \eqref{min_prob_1D_2D} in various ways. In our formulation, we will first solve part of the problem \eqref{min_prob_1D_2D} that can be written as  a linear least-squares problem in $\bw$, namely
{
	\begin{align}
	\hspace{-15mm}\min\limits_{\bw} \left\{ \rho_1 \left\| \IL \bw + 
	\bhh \right\|_2^2 \right. \hspace{40mm} \nonumber \\  \hspace{5mm} + \rho_2\left( \left\|\Lonetwo \bw +  \got \right\|_2^2
	+   \left\| \Ltwoone \bw +  \gto \right\|_2^2 \right). \label{lin_prob}
	\end{align}
}
The optimization problem \eqref{lin_prob} is a classical linear least-squares problem:
\begin{equation}\label{sol_lin_prob}
\tilde{\bw} = {\displaystyle {\underset {\bw}{\operatorname {arg\,min} }}}
\left\|~ 
\left[ \begin{array}{l}
\rho_1 \IL \\ \rho_2\Lonetwo \\ \rho_2\Ltwoone
\end{array} \right]  \bw + \left[ \begin{array}{l}
\bhh \\ \got \\ \gto 
\end{array} \right]~
\right \|_2.
\end{equation}
Using  $\tilde{\bw}$, we further relax the 
last term in \cref{min_prob_1D_2D}
as 
\begin{align} \nonumber
\rho_3 \Vert \Ltwod ( {\bw} \otimes \bw) &+ \bU \bw + \gtt \Vert_2^2  \longrightsquigarrow \\ 
&\rho_3 \Vert \Ltwod ( \tilde{\bw} \otimes \bw) + \bU \bw + \gtt \Vert_2^2.
\label{lastrelax}
\end{align} 
Using  $\Ltwod ( \tilde{\bw} \otimes \bw) = \Ltwod ( \tilde{\bw} \otimes \bI) \bw $, we rewrite
\eqref{lastrelax} as
\begin{align}\label{min_prob_1D_2D_2}
\hspace{-2mm} \rho_3 \Vert \Ltwod ( \tilde{\bw} \otimes \bw) + \bU \bw + \gtt \Vert_2^2  =
\rho_3
\Vert \IT \bw + \gtt
\Vert_2^2,
\end{align} 
where the matrix $\IT \in \mathbb{C}^{(N_s-n)^2 \times n}$ is defined as follows
\begin{equation}\label{formula_T}
\IT = \Ltwod ( \tilde{\bw} \otimes \bI)  + \bU.
\end{equation}
Then, using  \cref{min_prob_1D_2D_2}
in place of the last term in 
\eqref{min_prob_1D_2D},  we obtain a minimization problem that is now a linear  \ls problem. Thus, the solution to our final approximation to
\eqref{min_prob_1D_2D} is given by
\begin{equation}\label{w_new}
{\bw}_{\star} = {\displaystyle {\underset {\bw}{\operatorname {arg\,min} }}}
\left \Vert ~\left[ \begin{array}{l}
\rho_1\IL \\ \rho_2 \Lonetwo  \\\rho_2 \Ltwoone  \\ \rho_3 \IT
\end{array} \right] \bw +\left[ \begin{array}{l}
\bhh \\  \got \\ \gto \\ \gtt
\end{array} \right]~  \right \Vert_2. 
\end{equation}
Therefore, in the $n$th step of \lqoaaa, the optimization problem \eqref{truenls}  is relaxed and the solution of this relaxed problem (the weights) is given by \eqref{w_new}. The algorithms proceeds with the updated weights as we discuss next.

\subsection{Partition update via the greedy selection}   \label{sec:greedy}
Given the partition \eqref{spartrepeat} in the Step $n$ of the algorithm, 
\Cref{sec:computew} showed how to choose the barycentric weights 
$\bw$ to minimize a joint \ls measure over the uninterpolated data set. The only remaining component of the proposed approach is, then, to 
choose the next support point $\xi_{n+1}$ and 
update the data partition  \eqref{spartrepeat} (so that we repeat \Cref{sec:computew} for the updated partition until a desired tolerance achieved.)
In other words, we will move one sampling point from the \ls set 
$\boldsh$ to the interpolation set $\bolds$. Which point to move from
$\boldsh$ to $\bolds$ will be done in a greedy manner. 
To emphasize the iterative nature of the overall algorithm,  at this Step $n$ of the algorithm,
we will denote by $r_1^{(n)}(s)$ and $r_2^{(n)}(s,z)$ the two transfer functions of the current \lqo approximant. (Note that we dropped the superscript in   \Cref{sec:computew}  to simplify the notation there.)

We start by  defining two constants based on the data:
\begin{align}\label{def_M1_M2}
M_1 &= \max\limits_{s \in \Omega} \vert H_1(s)  \vert, \quad
M_2 = \max\limits_{s\in \Omega, z \in \Omega} \vert H_2(s,z) \vert.
\end{align}
For the current approximant in Step $n$, introduce the absolute error measures, deviations in the linear and quadratic parts:
\begin{align}\label{rel_err_meas_step_n}
\begin{split}
\epsilon_1^{(n)} &= \max\limits_{s \in \Omega} \vert H_1(s) - r_1^{(n)}(s) \vert, \\
\epsilon_2^{(n)} &= \max\limits_{s,z \in \Omega} \vert H_2(s,z) - r_2^{(n)}(s,z) \vert.
\end{split}
\end{align}
The next support point $\s_{n+1}$ is chosen
by means of a greedy search over the set  $\Omega \setminus \{\s_{1},\ldots,\s_n\}$ 
using the error measures $\epsilon_1^{(n)}$ and 
$\epsilon_2^{(n)}$. More precisely, if $\epsilon_1^{(n)}/N>\epsilon_2^{(n)}/N^2$, then
$\s_{n+1} = {\displaystyle {\underset {s\in \Omega}{\operatorname {arg\,max} }}\,\vert H_1(s) - r_1^{(n)}(s) \vert}$. On the other hand, if $\epsilon_1^{(n)}/N<\epsilon_2^{(n)}/N^2$, define $s^{(n+1)}$ and 
$z^{(n+1)}$  using 
$$
(s^{(n+1)},z^{(n+1)})= {\displaystyle {\underset {s,z\in \Omega}{\operatorname {arg\,max} }}\,\vert H_2(s,z) - r_2(s,z)} \vert.
$$
Now the question is whether choose $s^{(n+1)}$  or $z^{(n+1)})$ as $\s_{n+1}$. If only one of them was already a support point, then we choose the other one as $\s_{n+1}$.
If neither $s^{(n+1)}$ nor $z^{(n+1)}$ was previously chosen as a support point, then we compare $\vert H_1(s^{(n+1)}) - r_1^{(n)}(s^{(n+1)}) \vert$ and  $\vert H_1(z^{(n+1)}) - r_1^{(n)}(z^{(n+1)}) \vert$, and choose $\s_{n+1}$ as the one that yields the higher deviation in the first transfer function. Clearly, both cannot be already a support point due to the interpolation property. 

\begin{remark}
	{Instead of considering the full grid of pairs of sampling points $(s,z)$ and the associated measurements, we could consider a sparser grid for $H_2(s,z)$ samples. This modification would require  changing the greedy selection scheme accordingly to make sure that all possible combinations of selected points appear in the sparser grid. We skip this aspect in our examples and work with the full data set.}
\end{remark}

\subsection{The proposed algorithm: \textsf{AAA-LQO}}
\label{sec:lqoaaa}
Now, we have all the pieces to describe the algorithmic framework for the proposed method \textsf{AAA-LQO}, the \aaa algorithm for \lqo systems. 

Given the full \lqo data \cref{H1fulldata} and \cref{H2fulldata}, 
we initiate the approximant ($n=0$) by choosing  $r_1^{(0)}(s)$ as the average of $H_1(s)$ samples and  $r_2^{(0)}(s,z)$ as the average of $H_2(s,z)$ samples. 
Then, using the greedy selection strategy of \Cref{sec:greedy} we update the partition  \eqref{spartrepeat} and solve for the barycentric weights as in \Cref{sec:computew}, more specifically using  \eqref{w_new}. 
Let $n_{\textrm max}$ denote the largest dimension permitted for the data-driven
\lqo approximant $\widehat{\Sigma}_{\lqo}$ and 
and let $\epsilon$ denote the relative error tolerance.
Then, \textsf{AAA-LQO} terminates either when the prescribed dimension $n_{\textrm max}$ is reached, or when the prescribed error tolerance is achieved, namely
\begin{equation} \label{end_condition}
\max(   \epsilon_1^{(n)}/M_1,    \epsilon_2^{(n)}/M_2) < \epsilon.
\end{equation}
In \Cref{sec:numerics}, we depict the evolution of $ \epsilon_1^{(n)}/M_1$ and $ \epsilon_2^{(n)}/M_2$ during the \textsf{AAA-LQO} iterations. A sketch of \textsf{AAA-LQO} is given in \Cref{LQO_AAA_alg}.
\begin{algorithm}[htp] 
	\caption{\textsf{AAA-LQO}: \aaa algorithm for \lqo systems}  
	\label{LQO_AAA_alg}                                     
	\algorithmicrequire~\\
	\hspace*{3mm} Sampling points $\{s_1,\ldots,s_{N_s}\}$, and 
	samples $\{H_1(s_i)\}$ and
	$\{H_2(s_i,s_j)\}$ of an
	\lqo system; \\
	\hspace*{3mm} Maximum dimension  allowed $n_{\textrm max}$; \\
	\hspace*{3mm} Stopping tolerance $\epsilon$.\\
	
	\algorithmicensure~\\
	\hspace*{3mm} data-driven \lqo system $\widehat{\Sigma}_{\lqo}$ as in  \cref{redLQO}.\\[1ex]
	\hspace*{2mm}{\small 0:}
	$n=0$,
	$r_1^{(0)} =
	\textsf{avg}\{ H_1(s_i)\}$, and $r_2^{(0)} = \textsf{avg}\{H_2(s_i,s_j) \}$
	\\
	\algorithmicwhile~   $\max(   \epsilon_1^{(n)}/M_1,    \epsilon_2^{(n)}/M_2) > \epsilon \  \ \text{and} \ \ n<r_{\textrm max}$  \label{tol_cond}
	\begin{algorithmic} [1] 
		\vspace{-2mm}
		\STATE Employ the greedy selection scheme to choose the next support point(s)
		and update the partitioning
		as described in \Cref{sec:greedy}.
		\STATE Compute the vector of weights $\bw_{\star}$ as in \cref{w_new}.
		\STATE Update $r_1^{(n)}(s)$ and $r_2^{(n)}(s,z)$, and
		compute the errors $\epsilon_1^{(n)}$ and $\epsilon_2^{(n)}$ as in
		\eqref{rel_err_meas_step_n}.
		\STATE $n =n+1$
		\cref{rel_err_meas_step_n}.\\
		\hspace*{-3.25ex}\algorithmicend
	\end{algorithmic}
\end{algorithm}

\begin{remark}
	Note that, by choosing complex-conjugate sampling points and sampled values, one can enforce the fitted models to be real-valued. This is actually enforced for both examples presented in \Cref{sec:numerics}.
\end{remark}

\section{Numerical examples}\label{sec:numerics}

We test \lqoaaa, as given in \Cref{LQO_AAA_alg}, on two \lqo systems. We also apply the original \aaa algorithm (from the linear case) to the data corresponding to the first (linear) transfer function only.  Therefore, we construct two approximants: (1) A data-driven \lqo approximant of order-$n$ using \lqoaaa and (2) A data-driven linear approximant using \aaa.
Note that both approximants are real-valued, enforced by using a data set that is closed under complex conjugation.

\subsection{Example 1}\label{sec:iss_ex}

First, we use a single-input/single-output version of the ISS {1R Model}  from the SLICOT MOR benchmark collection \cite{morChaV02}.  We construct a \lqo system from this linear model by adding a quadratic output with the choice of  $\bM = 0.6 \bI_{270}+ 0.3\bI^{(-1)}_{270}+0.3\bI^{(+1)}_{270} \in \mathbb{R}^{270 \times 270}$, which scales the product of the state variable with itself, in the output equation. Here, $\bI^{(k)}_{270}$ denotes a quasi-diagonal matrix for which the  entries of ones are shifted from the main diagonal based on the integer $k$ ($k>0$ stands for upper shifting, while $k<0$ is used for lower shifting - also, note that $\bI^{(0)}_{270}= \bI_{270}$).

We collect the following data: pick $60$
logarithmically-spaced points in the interval 
$[10^{-1},10^2] \iop$ and add its conjugate pairs in $[-10^{-2},10^1)] \iop$ to have $N_s=120$ sampling points $\{s_i\}$ and the samples $\{H_1(s_i)\}$ for $i=1,2,\ldots,N_s$
as in \eqref{H1fulldata}.  Then, as in 
\eqref{H2fulldata}, we sample the second-transfer function at $H_2(s_i,s_j)$ for $i,j=1,2,\ldots,N_s$. The sampled data are depicted in \Cref{fig:1}, where we  display  the measurements evaluated only on the ``positive side" of the imaginary axis and skipping the conjugate data. 
\begin{figure}[h]
	\hspace{-6mm}
	\includegraphics[scale=0.235]{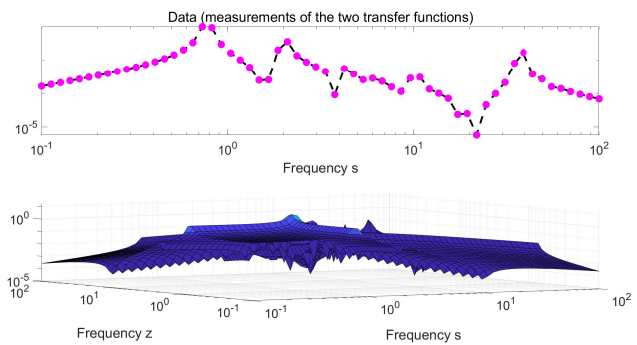}
	\vspace{-4mm}
	\caption{Measurements corresponding to the transfer functions; $H_1(s)$ (top) and $H_2(s,z)$ (bottom).}	\label{fig:1}
	\vspace{-2mm}
\end{figure}

We apply \Cref{LQO_AAA_alg}  with $n_{\textrm max}=30$ 
and $\epsilon = 10^{-2}$ (relative tolerance value corresponding to $99\%$ approximation error on the data).
With these variables, \lqoaaa yields a data-driven \lqo model of order $n=18$. 

Using only the $\{ H_1(s_i)\}$ samples (corresponding to the linear observation map), we  apply \aaa and obtain a data-driven linear approximant of order $n=18$. The \aaa approximant is constructed to simply illustrate that a linear dynamical system approximation is not sufficient to accurately represent the underlying \lqo system.

In the top plot of \Cref{fig:2}, we show the magnitude of the first transfer function $H_1(s)$ of the original system together with that of the linear \aaa model and the first transfer function
($r_1^{(n)}(s)$) of the \lqoaaa model. As expected, \aaa model does a good job in matching 
the linear part of the output. Similarly, the \lqoaaa model also
matches $H_1(s)$ accurately. To better illustrate this, 
in the bottom plot of \Cref{fig:2}, we depict the magnitude of the approximation errors in $H_1(s)$. The plot reveals that the \lqoaaa model has a smaller error for most of the frequency values, even in approximating  $H_1(s)$. This happens despite the fact that it focuses on both $H_1(s)$ and $H_2(s,z)$ unlike the \aaa model, which only tries to approximate $H_1(s)$.

\begin{figure}[h]
\hspace{-6mm}
		\includegraphics[scale=0.24]{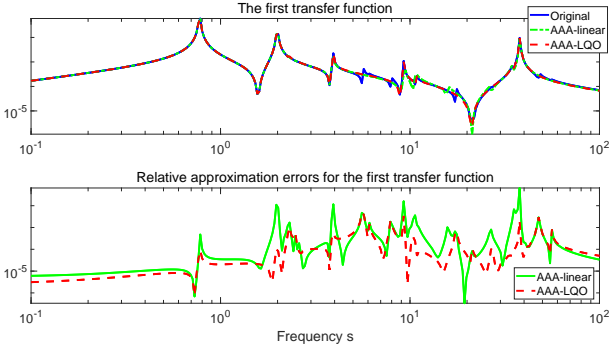}
	\vspace{-4mm}
	\caption{First transfer function approximation.}
	\label{fig:2}
	\vspace{-2mm}
\end{figure}

In \Cref{fig:3} we depict the selected support points (interpolation points) for both \aaa and \lqoaaa algorithms (without the complex conjugate pairs), as well as the poles of the learned models (i.e., the eigenvalues of $\bA_r$ in both cases). Note that there are  $9$ complex conjugate pairs of support points for each method. Even though some of the support points of \aaa and \lqoaaa overlap, two of the pairs are different. This difference causes a big deviation in the  the pole pattern as shown in the bottom plot, illustrating  that even the linear part of the \lqoaaa approximant, i.e., $r_1^{(n)}(s)$, is fundamentally different than the linear \aaa model. This is expected since \lqoaaa constructs $r_1^{(n)}(s)$ and  $r_2^{(n)}(s,z)$ together by minimizing a joint \ls measure in both $H_1(s)$ and $H_2(s,z)$.

\begin{figure}[h]
	\hspace{-6mm}		\includegraphics[scale=0.24]{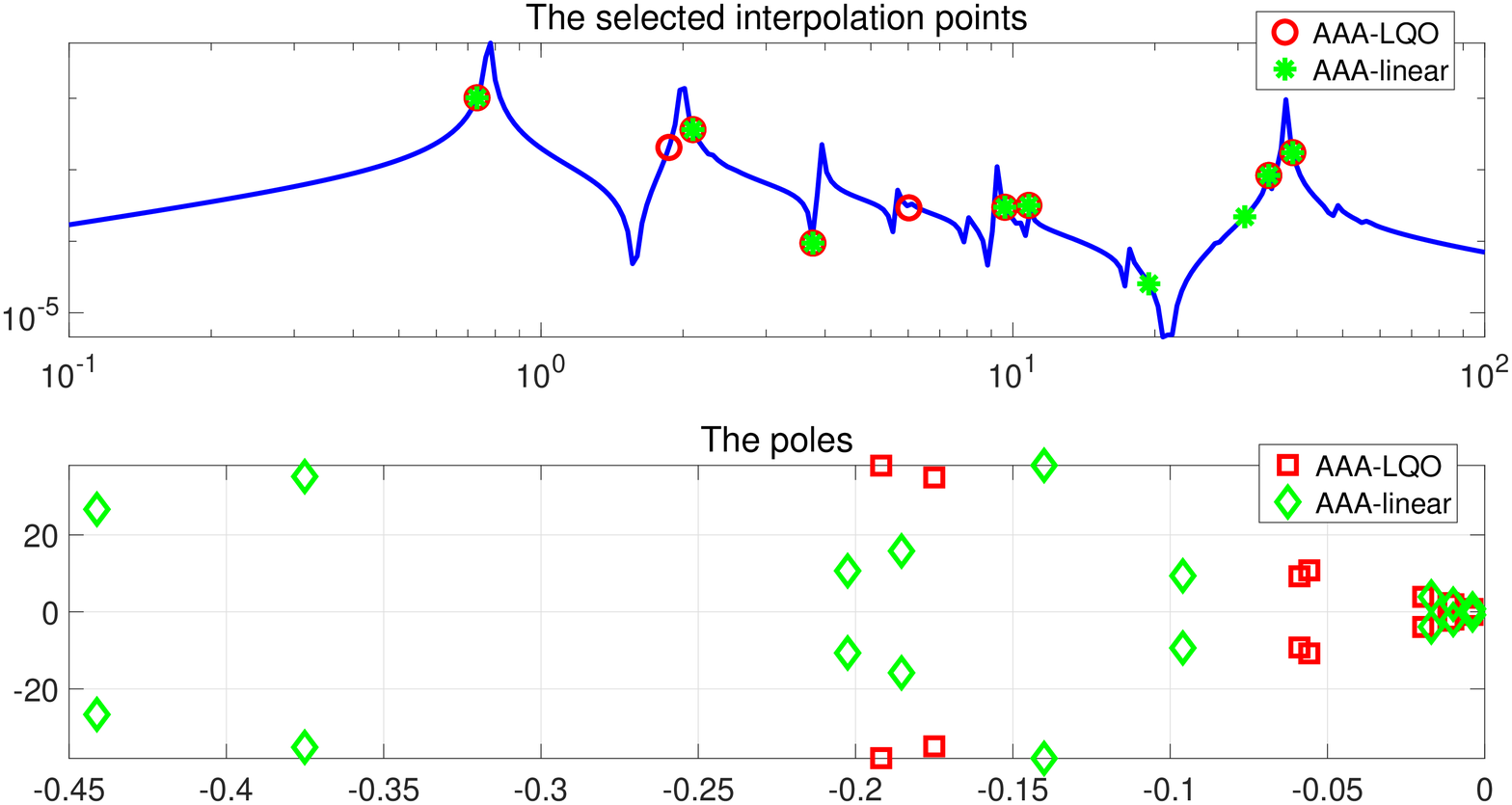}
	\vspace{-4mm}
	\caption{The support points (top) and the poles (bottom) for the two \aaa reduced-order models.}
	\label{fig:3}
	\vspace{-2mm}
\end{figure}

To show the overall performance of \lqoaaa in accurately approximating not only
$H_1(s)$ but also $H_2(s,z)$ (the full \lqo behavior), we perform a time-domain simulation of  the original \lqo system $\Sigma_{\lqo}$, the data-driven \lqoaaa model $\widehat{\Sigma}_{\lqo}$, and the linear \aaa model  by using 
$u(t) = 0.5 \cos(4 \pi t)$ as the control input. During the simulation of the original system $\Sigma_{\lqo}$, we also compute  only the linear part of the output, which the \aaa model should approximate well. The results are given in the top plot of \Cref{fig:4}. The first observation is that the output of $\widehat{\Sigma}_{\lqo}$ from \lqoaaa accurately replicates the output of $\Sigma_{\lqo}$. On the other hand, the linear \aaa model completely misses the quadratic output and is only able to approximate the linear component in the output, as expected. The approximation error in the output corresponding to $\widehat{\Sigma}_{\lqo}$ is depicted in the bottom plot of  \Cref{fig:4}.

\begin{figure}[h]
	\hspace{-6mm}
		\includegraphics[scale=0.24]{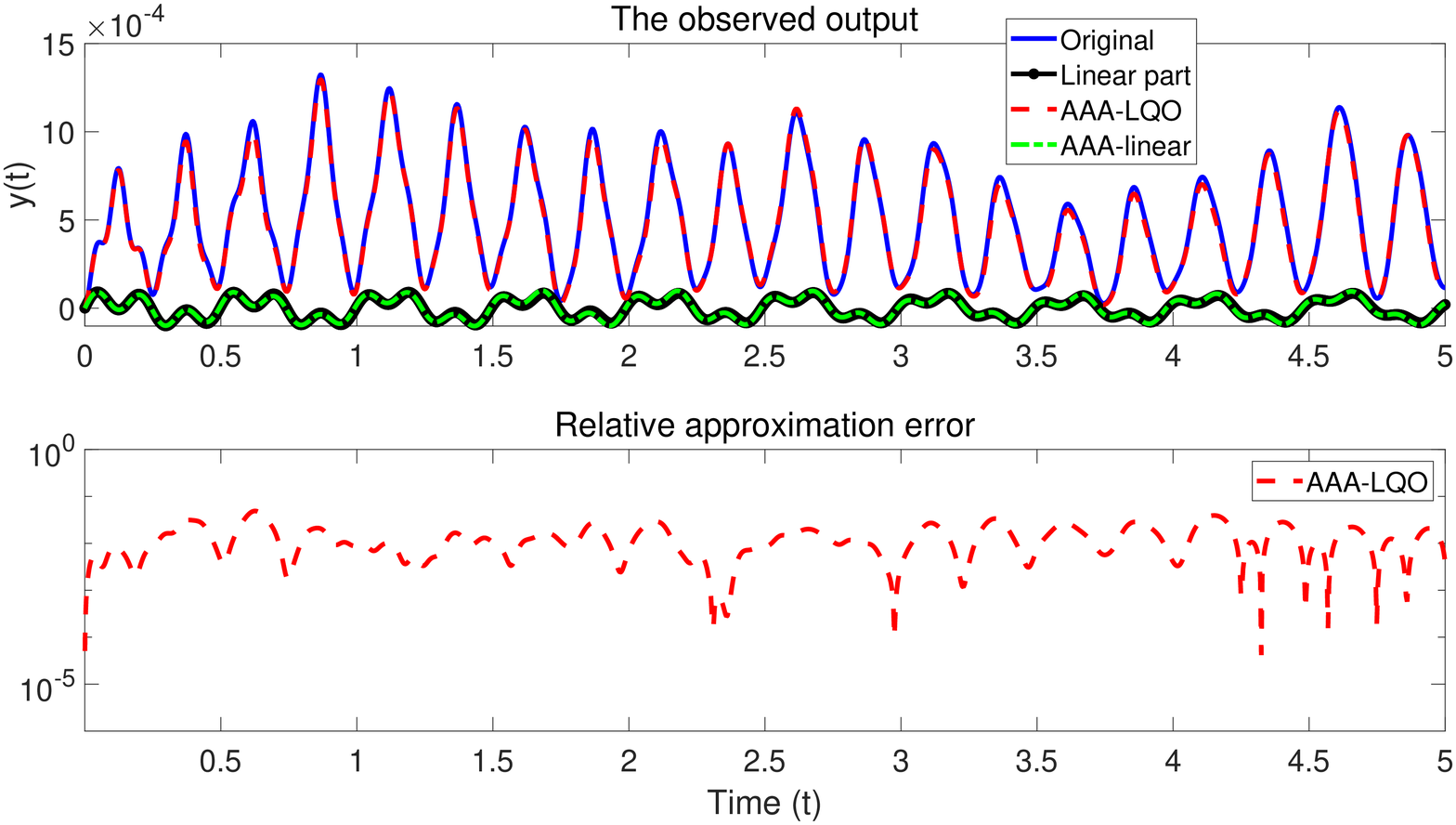}
	\vspace{-4mm}
	\caption{Time-domain simulations: (top) output of the original and data-driven models, (bottom) approximation error.}
	\label{fig:4}
	\vspace{-2mm}
\end{figure}

In \Cref{fig:5} we show the convergence behavior of \lqoaaa by plotting the evolution of the relative approximation errors ($\epsilon_1^{(n)}/M_1$ and $\epsilon_2^{(n)}/M_2)$ for all even values of $n$. For a reference, we also depict the convergence behavior of \aaa. The figure illustrates that after $n=18$, both relative errors fall below the given tolerance of $10^{-2}$ and the algorithm terminates.

\begin{figure}[h]
	\hspace{-6mm}
			\includegraphics[scale=0.24]{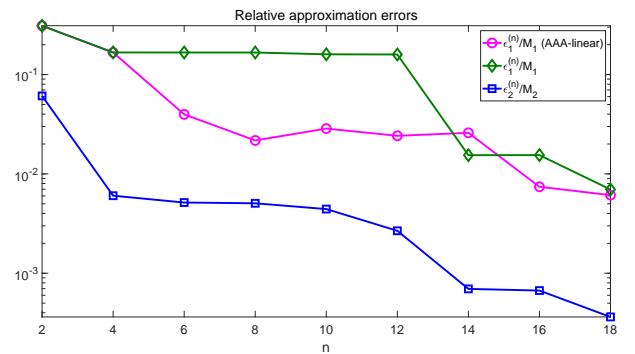}
	\vspace{-4mm}
	\caption{Relative approximation errors in each step.}
	\label{fig:5}
	\vspace{-2mm}
\end{figure}

To investigate how the order of the \lqoaaa model varies based on the stopping tolerance, we set $n_{\textrm max} = 100$ and run 
\lqoaaa for four tolerance  values $\tau = 10^{-2}$,
$\tau = 10^{-3}$, $\tau = 10^{-4}$, and $\tau = 10^{-5}$.
The results are displayed in \Cref{tab:1}. 
For the case of $\tau = 10^{-5}$,
in \Cref{fig:55} we depict the 
convergence behavior of \lqoaaa by plotting
$\epsilon_1^{(n)}/M_1$ and $\epsilon_2^{(n)}/M_2$ during the iteration.

\begin{table}[hh] 	
	\begin{center}
		\begin{tabular}{||c | c | c | c||} 
			\hline
			$\tau = 10^{-2}$ &  $\tau = 10^{-3}$ &  $\tau = 10^{-4}$ &  $\tau = 10^{-5}$ \\ [0.5ex] 
			\hline\hline
			$n=18$ &  $n=28$ &  $n=56$ &  $n=62$ \\ 
			\hline
		\end{tabular}
		\caption{Tolerance values $\tau$ vs. the order $n$}
		\label{tab:1}
	\end{center}
\end{table}

\begin{figure}[h]
	\hspace{-6mm}
			\includegraphics[scale=0.24]{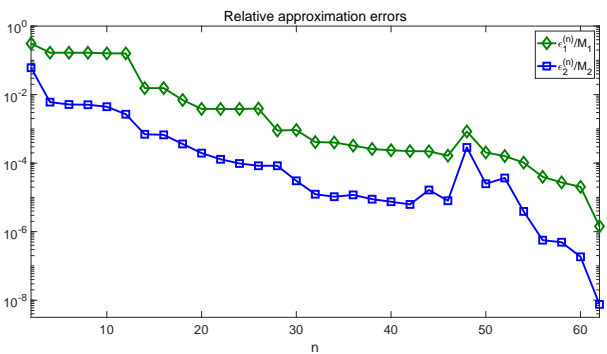}
	\vspace{-4mm}
	\caption{Relative approximation errors in each step.}
	\label{fig:55}
	\vspace{-2mm}
\end{figure}

\subsection{Example 2}\label{sec:galerkin_ex}

This model taken from \cite{PN19} corresponds to an \lqo system whose output measures a variance in the state-variable. A linear mass-spring-damper SISO dynamical system was modified in \cite{Pulch18} by means of stochastic modeling, by replacing the physical parameters by independent random variables, yielding a linear dynamical system with multiple outputs. Based on this multiple output system, a SISO \lqo system was derived in \cite{PN19} where the output corresponds to the variance of tne original output (and thus is quadratic in nature). We refer the reader to \cite{PN19} for further details. We obtain the measurements from a version of this model corresponding to an underlying \lqo system of order $\cN = 960$.

The main difference from the previous example is that in this model the observed output does not have a linear component and depends on the state variable solely quadratically, i.e., $\bc = \bfz$ 
in \cref{eq:def_LQO}. Hence, $H_1(s) = 0, \ \forall s$.

As sampling points $\{s_i\}$, we choose $60$ logarithmically spaced points over the interval $[10^{-1},10^1]\iop$ together with its conjugate pairs, leading to $N_s = 120$ samples. Since $H_1(s) =0$, we only need to sample $H_2(s_i,s_j)$ for $i,j=1,2,\ldots,N_s$. 
The corresponding data for the second transfer function are depicted in \Cref{fig:6}.

\begin{figure}[h]
	\hspace{-6mm}	
		\includegraphics[scale=0.235]{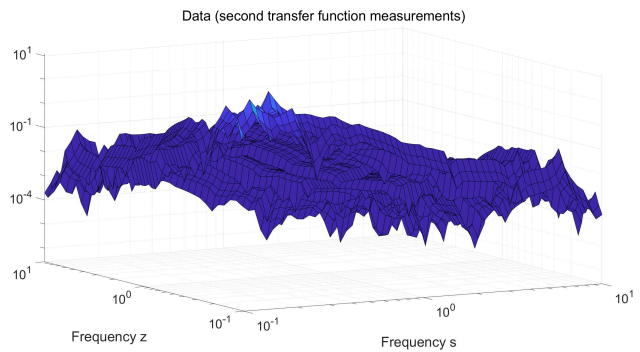}
	\vspace{-4mm}
	\caption{Measurements of the second transfer function.}	\label{fig:6}
	\vspace{-2mm}
\end{figure}

We apply \lqoaaa with $n_{\textrm max}=50$ 
and $\epsilon = 10^{-3}$ (relative stopping criterion), obtaining a \lqo model of order $n=30$. To show the accuracy of the approximant, we perform time-domain simulations of the full  model and the approximant with the input $u(t) =  \sin(0.2  t)$. We depict the observed outputs in the top plot  of \Cref{fig:7}, illustrating 
an accurate approximation. The corresponding output error is plotted in the bottom plot of \Cref{fig:7}.

\begin{figure}[h]
	\hspace{-6mm}		\includegraphics[scale=0.24]{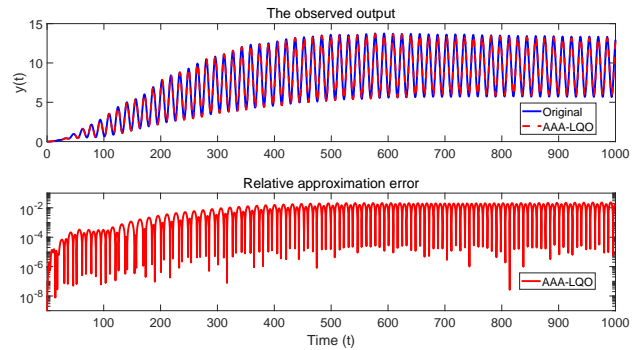}
	\vspace{-4mm}
	\caption{Time-domain simulations; output of the original and the reduced system (up) + approximation error (down).}
	\label{fig:7}
	\vspace{-2mm}
\end{figure}

Finally, in \Cref{fig:8} we show the convergence behavior 
of \lqoaaa by plotting the evolution of 
relative approximation error $\epsilon_2^{(n)}/M_2$.

\begin{figure}[h]
	\hspace{-6mm}			\includegraphics[scale=0.24]{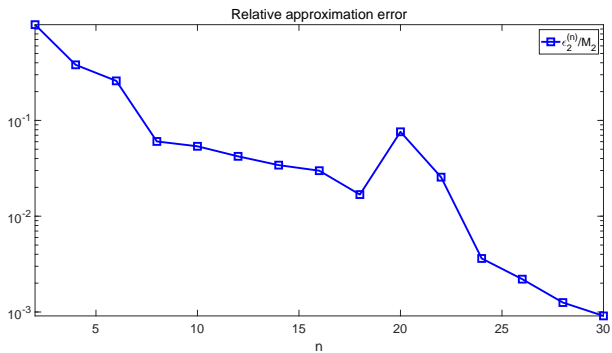}
	\vspace{-4mm}
	\caption{Maximum relative approximation error at each step.}
	\label{fig:8}
	\vspace{-2mm}
\end{figure}

\begin{remark}
	Since \lqoaaa uses a greedy selection scheme and is not a descent algorithm, there is no theoretical guarantee that the maximum approximation error will decrease monotonically.
	This can be seen in \Cref{fig:5,fig:55,fig:8}.  This behavior was also observed in the original \aaa algorithm; see, e.g, Application 6.3 in \cite{NST18}. However, numerically
	the error indeed decreases monotonically with $n$ in most cases.
\end{remark}

\section{Conclusions}
\label{sec:conc}

We have proposed a novel data-driven modeling method, called \lqoaaa,  for linear systems with quadratic outputs (\lqo). 
\lqoaaa extends the \aaa algorithm to this new setting by first developing the barycentric representation theory for the two transfer functions arising in the analysis of \lqo and  then 
formulating a \ls minimization framework to efficiently solve for the barycentric coefficients. 
The two numerical examples illustrate that \lqoaaa provides high-fidelity data-driven approximants to the original model.

The barycentric form we developed here for \lqo systems offers promising research directions    
for modelling systems with general polynomial observation maps, as well as for nonlinearities appearing in the dynamical equation such as bilinear or quadratic-bilinear systems. These topics are the focus of on-going research.

\section{Acknowledgements}
The authors would like to thank Dr. Akil Narayan for providing the source codes for generating the numerical example presented in \Cref{sec:galerkin_ex}.\\
S. Gugercin was supported in parts by National Science Foundation under Grant No. 
DMS-1720257 and DMS-1819110.
Part of this material is  based upon work supported by the National Science 
Foundation under Grant No. DMS-1439786 and by the Simons Foundation Grant No. 
507536  while Gugercin was in residence at the 
Institute for Computational and Experimental Research in Mathematics in 
Providence, RI, during the \textquotedblleft{}Model and dimension reduction in 
uncertain and dynamic systems\textquotedblright{} program.

\addcontentsline{toc}{section}{References}
\bibliographystyle{plainurl}
\bibliography{LQO_AAA_ref}             
\appendix
\section{Proof of Lemma \ref{Lemma2} } \label{sec:appendix}

Substitute $s = \s_i$ and $z = \sh_j$ into (\ref{formulae_n2_d2}) to obtain
\begin{align*}
n_2(\s_i,\sh_j)
&= \displaystyle \sum_{k=1}^n \sum_{\ell=1}^n  h_{k,\ell}w_k w_\ell P_{k,\ell}(\s_i,\sh_j) \\
&= \displaystyle  \sum_{\ell=1}^n  h_{i,\ell}w_i w_\ell P_{i,\ell}(\s_i,\sh_j).
\end{align*}
and
\begin{align*}
 d_2(\s_i,\sh_j) &= P(\s_i,\sh_j)+\displaystyle \sum_{k=1}^n w_k p_k(\s_i) p(\sh_j) \\ &+ \displaystyle \sum_{\ell=1}^n w_\ell p_\ell(\sh_j) p(\s_i) +\displaystyle \sum_{k=1}^n \sum_{\ell=1}^n w_k w_\ell P_{k,\ell}(\s_i,\sh_j) \\
&= w_i p_i(\s_i) p(\sh_j)+\sum_{\ell=1}^n  w_i w_\ell P_{i,\ell}(\s_i,\sh_j),
\end{align*}
where $P,p_k$, and $P_{k,l}$ are as given in \eqref{eq:definepij}.
Hence, write $ r_2(\s_i,\sh_j) = \frac{n_2(\s_i,\sh_j)}{d_2(\s_i,\sh_j)}$  as
\begin{align}
\begin{split}
r_2(\s_i,\sh_j) &= \frac{\displaystyle  \sum_{\ell=1}^n  h_{i,\ell}w_i w_\ell P_{i,\ell}(\s_i,\sh_j)}{w_i p_i(\s_i) p(\sh_j)+\displaystyle \sum_{\ell=1}^n  w_i w_\ell P_{i,\ell}(\s_i,\sh_j)}.
\end{split}
\end{align}
Introduce the notation
\begin{equation}
P^L_{i}(\s_i,z) = \prod_{k=1, k \neq i}^n \prod_{\ell=1}^n (\xi_i-\xi_k) (z-\xi_\ell) =  p_i(\s_i) p(z).
\end{equation}
Since $P^L_{i}(\s_i,\sh_j) = P_{i,\ell}(\s_i,\sh_j) (\sh_j-\s_\ell)$ holds, 
we can write 
\begin{align}
\begin{split}
r_2(\s_i,\sh_j) &= \frac{\displaystyle  \sum_{\ell=1}^n  h_{i,\ell}w_i w_\ell \frac{P^L_{i}(\s_i,\sh_j)}{\sh_j-\s_\ell}}{w_i P^L_{i}(\s_i,\sh_j)+\displaystyle \sum_{\ell=1}^n  w_i w_\ell \frac{P^L_{i}(\s_i,\sh_j)}{\sh_j-\s_\ell}}.
\end{split}
\end{align}
By simplifying $w_i P^L_{i}(\s_i,\sh_j)$ from both the numerator and the denominator in the above expression proves the first desired result in \eqref{eval_r2_s_sh}. The proof for $r_2(\sh_j,\s_i)$ follows similarly.

\end{document}